\documentclass[11pt]{article}
\usepackage{latexsym, amssymb, enumerate, amsmath, amsthm,multicol}
\usepackage[utf8]{inputenc}
\usepackage[T1]{fontenc}
\usepackage{fullpage}
\usepackage{subfigure}
\usepackage{layout}

\usepackage{cleveref}
\usepackage{cite}
\usepackage{url}
\usepackage{esint}
\usepackage{dsfont}
\usepackage{color}
\usepackage{overpic}
\usepackage{soul}

\usepackage{graphicx,epsfig}

\crefname{assumption}{assumption}{assumptions}
\crefname{thm}{theorem}{theorems}
\crefname{lem}{lemma}{lemmas}
\crefname{cor}{corollary}{corollaries}
\crefname{prop}{proposition}{propositions}
\Crefname{theorem}{Theorem}{Theorems}
\crefname{conjecture}{conjecture}{conjectures}

\newcommand{\R}{\mathbb{R}}

\renewcommand{\epsilon}{\varepsilon}
\newcommand{\eps}{\varepsilon}

\let\phi\varphi

\newtheorem{thm}{Theorem}[section]
\newtheorem{defi}{Definition}[section]
\newtheorem{prop}[thm]{Proposition}
\newtheorem{lem}[thm]{Lemma}

\newtheorem{definition}[thm]{Definition}
\newtheorem{proposition}[thm]{Proposition}

\newtheorem{example}[thm]{Example}

\newtheorem{rem}[thm]{Remark}
\newtheorem{theorem}[thm]{Theorem}

\newcommand{\NC}[1]{\textcolor{black}{#1}}

\date\today
\author{Emeric Bouin 
\footnote{CEREMADE - Universit\'e Paris-Dauphine, UMR CNRS 7534, Place du Mar\'echal de Lattre de Tassigny, 75775 Paris Cedex 16, France. E-mail: \texttt{bouin@ceremade.dauphine.fr}}\and
Nils Caillerie
\footnote{Institut Camille Jordan (ICJ), Université Claude Bernard Lyon 1, 43 boulevard du 11 novembre 1918, 69622 Villeurbanne Cedex, France. E-mail: \texttt{caillerie@math.univ-lyon1.fr}}
}

\begin{document}
\title{Spreading in kinetic reaction-transport equations \\in higher velocity dimensions}
\maketitle

\begin{abstract}
In this paper, we extend and complement previous works about propagation in kinetic reaction-transport equations. The model we study describes particles moving according to a velocity-jump process, and proliferating according to a reaction term of monostable type. We focus on the case of bounded velocities, having dimension higher than one. We extend previous results obtained by the first author with Calvez and Nadin in dimension one. We study the large time/large scale hyperbolic limit via an Hamilton-Jacobi framework together with the half-relaxed limits method. We deduce spreading results and the existence of travelling wave solutions. A crucial difference with the mono-dimensional case is the resolution of the spectral problem at the edge of the front, that yields potential singular velocity distributions. As a consequence, the minimal speed of propagation may not be determined by a first order condition.
\end{abstract}
\noindent{\bf AMS Class. No:} {35Q92, 45K05, 35C07}

\section{Introduction}

\subsection*{The model.}

In this paper, we are interested in propagation phenomena occuring in the following reaction-transport equation
\begin{equation} \label{eq:main}
\begin{cases}
\partial_t f(t,x,v) + v \cdot \nabla_x f(t,x,v) = M(v) \rho(t,x) - f(t,x,v) + r\rho(t,x)  \left( M(v) - f(t,x,v) \right),\\ \hfill (t,x,v) \in \R_+ \times \R^n \times V\, ,&\smallskip\\
f(0,x,v) = f_0(x,v)\,, \hfill (x,v) \in   \R^n \times V\, ,
\end{cases}
\end{equation}
where $r>0$. The mesoscopic density $f$ depends on time $t\in\R^+$, position $x\in\R^n$ and velocity $v\in V$ and describes a population of individuals. The macroscopic density is $\rho(t,x) = \int_V f(t,x,v)\, dv$. The subset $V\subset \R^n$ is the set of all possible velocities. From now on, we assume
\begin{itemize}
\item[{\bf(H0)}]\label{H0} The velocity set $V \subset \R^n$ is compact.
\end{itemize}
For any given direction $e \in \mathbb{S}^{n-1}$, we define
\begin{equation*}
\overline{v}(e) = \max\left\lbrace v\cdot e, v\in V\right\rbrace, \qquad \mu(p) = \vert p \vert \overline{v}\left( \frac{p}{\vert p \vert} \right), \quad \mathrm{Arg}\,\mu(p) = \left\{ v\in V\mid v\cdot p = \mu(p)\right\}.
\end{equation*}
We set
\begin{equation*}
v_{max}:=\underset{v\in V}{\mathrm{sup}}\,|v|, \qquad |V|:=\int_V dv.
\end{equation*}

Individuals move following a so-called velocity-jump process. That is, they alternate successively a run phase, with velocity $v\in V$, and a change of velocity at rate 1, which we call the tumbling. The new velocity is chosen according to the probability distribution $M$. Throughout the paper, we assume 
\begin{itemize}
\item[{\bf(H1)}]\label{H1} $M\in L^1(V)$, and
\begin{equation}\label{eq:hypM}
\left\langle v \right\rangle_M:=\int_{V} vM(v)dv = 0.
\end{equation}
\end{itemize}
Note that it is challenging to replace the linear BGK operator $M\rho - f$ by a more general collision operator of the form $P(f)-\Sigma f$ where $P$ is a positive operator. 
However, to remain consistent with \cite{bouin_propagation_2015}, we will stick to their framework and leave this question for future work.
\begin{rem}
In fact, our analysis can easily be extended to the case $\left\langle v \right\rangle_M\in \R^n \setminus {0}$. Setting $\mathbb{V}:=V-\left\langle v \right\rangle_M$, $\mathbb{M}(w):=M(w+\left\langle v \right\rangle_M)$ and $\mathbb{F}(t,x,w):=f(t,x+\left\langle v \right\rangle_Mt,w+\left\langle v \right\rangle_M)$, for all $(t,x,w)\in\R_+\times\R^n\times \mathbb{V}$, we recover our assumptions in the new framework.
\end{rem}

The reproduction of individuals is taken into account through a reaction term of monostable type. The constant $r>0$ is the growth rate in absence of any saturation. New individuals start with a velocity chosen at random with the same probability distribution $M$. The quadratic saturation term accounts for local competition between individuals, regardless of their speed. 

We assume that initially $0 \leq f_0 \leq M$, so that this remains true for all times, see \cite{bouin_propagation_2015,cuesta_traveling_2012}.

\subsection*{Earlier works and related topics}
 
It is relatively natural to address the question of spreading for \eqref{eq:main} since there is a strong link between \eqref{eq:main} and the classical Fisher-KPP equation \cite{fisher_wave_1937,kolmogorov_etude_1937}. Indeed, a suitable parabolic rescaling 
\begin{equation} \label{eq:main2}
\epsilon^2 \partial_t g_\eps +  \epsilon  v \cdot  \nabla_x g_\eps  = \left(M(v) \rho_{g_\eps} - g_\eps\right) + \epsilon^2  r\rho_{g_\eps} \left( M(v) - g_\eps \right)\,,
\end{equation}
leads to the Fisher-KPP equation (see \cite{cuesta_traveling_2012} for example) in the limit $\eps \to 0$, 
\begin{align}\label{eq:kolmogorov_etude_1937}
&\partial_t \rho^0-  \left\langle v^2 \right\rangle_M\partial_{xx} \rho^0 = r \rho^0 \left( 1 - \rho^0 \right)\, ,\\
&g^0 := \lim_{\eps \to 0} g_\eps = M \rho^0 \nonumber,
\end{align}
assuming that the two following conditions on $M$ hold:
\begin{equation*}
\int_V vM(v)dv=0,\quad \left\langle v^2 \right\rangle_M:=\int_V v^2 M(v)dv>0.
\end{equation*}
We recall that for nonincreasing initial data decaying sufficiently fast at $x = +\infty$, the solution of \eqref{eq:kolmogorov_etude_1937} behaves asymptotically as a travelling front moving at the minimal speed $c^* = 2 \sqrt{r\left\langle v^2 \right\rangle_M}$  \cite{kolmogorov_etude_1937,aronson_multidimensional_1978}. 
However, even though the philosophy of the results will be the same in spirit, we emphasize that nothing related to this parabolic limit will be used in the present paper. Our argumentation does not rely on any perturbative analysis. Hence, we obtain results without any smallness assumption on the parameters. This will yield significant differences, regarding both the results and the methods of proof.

A short review of earlier results is now in order. Hadeler has worked on propagation for reaction-telegraph equations \cite{Hadeler_1988,Hadeler1999}, that can be seen as two-speeds kinetic models.
Morever, a similar type of result was obtained by Cuesta, Hittmeir and Schmeiser  \cite{cuesta_traveling_2012} in the diffusive regime (\em i.e. \em for sufficiently small $\eps$ in \eqref{eq:main}). Using a micro-macro decomposition, they constructed possibly  oscillatory travelling waves of speed $c\geq 2\sqrt{rD}$ for $\epsilon$ small enough (depending on $c$). In addition, when the set of admissible speeds $V$ is bounded, $c> 2\sqrt{rD}$, and $\epsilon$ is small enough, they prove that the travelling wave constructed in this way is indeed nonnegative.

Propagation for the full kinetic model \eqref{eq:main} has then been investigated by the first author with Calvez and Nadin in \cite{bouin_propagation_2015}. In one dimension of velocities, and when the velocities are bounded, they proved the existence and stability of travelling waves solutions to \eqref{eq:main}. The minimal speed of propagation of the waves is determined by the resolution of a spectral problem in the velocity variable. In particular, it is not related with the KPP speed, except that the speeds coincide in the diffusive regime. It is worth mentioning that the case of unbounded velocities is significantly different as the front spreads with arbitrarily large speed \cite{bouin_propagation_2015}. This case shall not be discussed further in this paper. This phenomenon was newly appearing for this type of equations and unexpected from the macroscopic limit. One aim of this paper is to extend the construction of travelling waves solutions to any velocity dimension, which was left open after \cite{bouin_propagation_2015}.

There is a strong link between this KPP type propagation phenomena and large deviations for the underlying velocity-jump process. Indeed, it is well known that fronts in Fisher-KPP equations are so-called \textit{pulled fronts}, that is, are triggered by very small populations at the edge that are able to reproduce almost exponentially. Thus, studying large deviations for these type of processes at the kinetic level is an interesting problem in itself. In \cite{bouin_kinetic_2012,bouin_hamilton-jacobi_2015}, the authors have combined Hamilton-Jacobi equations and kinetic equations to study large deviations (and propagation) from a PDE point of view. These works show that he asymptotics of large deviations in the kinetic equation do not coincide with the asymptotic of large deviations obtained after a diffusive approximation.

As a side note, the Hamilton-Jacobi technique (that will be described in the next subsection) has also much been used recently to study long time dynamics in all sorts of stuctured models. An interested reader could describe the evolution of dominant phenotypical traits in a given population \NC{reading  \cite{barles_concentration_2009,lorz_dirac_2011,bouin_hamiltonjacobi_2015}} and the references therein), study different adaptative dynamics issues \cite{diekmann_dynamics_2005}, describe propagation in reaction-diffusion models of kinetic types \cite{bouin_invasion_2012} but also in age renewal equations\cite{calvez_limiting_2016}. This approach has also recently been used to study large deviations of velocity jump-processes\cite{bouin_kinetic_2012,bouin_large_2016,caillerie_large_2017} or slow-fast systems \cite{bressloff_path_2014,bressloff_hamiltonian_2014,faggionato_averaging_2010,kifer_large_2009,perthame_asymmetric_2009}.

\subsection*{The Hamilton-Jacobi limit}

After the seminal paper by Evans and Souganidis \cite{freidlin_functional_1985,evans_pde_1989}, an important technique to derive the propagating behavior in reaction-diffusion equations is to revisit the WKB expansion to study hyperbolic limits. We will directly present the technique on our problem for conciseness but one can find the original framework for the Fisher-KPP equation in \cite{evans_pde_1989} and complements in \cite{barles_solutions_1994,barles_wavefront_1990,souganidis_front_1997,crandall_users_1992}.

We perform the hyperbolic scaling $\left( t,x,v \right) \to \left( \frac{t}{\eps} , \frac{x}{\eps} ,v \right)$ in \eqref{eq:main}. Importantly, the velocity variable is not rescaled (it cannot be rescaled since it lies in a bounded set). The \textit{kinetic Hopf-Cole transformation} (already used in \cite{bouin_kinetic_2012,caillerie_large_2017}) is written 
\begin{equation}\label{eq:HopfColetransform}
\forall (t,x,v) \in \R^+ \times \R^n \times V, \qquad f^{\eps}(t,x,v) = M(v) e^{-\frac{\varphi^{\eps}(t,x,v)}{\eps}}.
\end{equation}

Thanks to the maximum principle \cite{cuesta_traveling_2012}, $\varphi^{\eps}$ is well defined and remains nonnegative for all times.
Plugging \eqref{eq:HopfColetransform} in \eqref{eq:main}, one obtains the following equation for $\varphi^{\eps}$:
\begin{equation}\label{eq:mainHJeps}
\partial_t \varphi^{\eps} + v \cdot \nabla_x \varphi^{\eps} + r = (1+r)\int_{V} M(v') \left( 1-e^{\frac{\varphi^\eps(v) - \varphi^\eps(v')}{\eps}} \right) dv'  +  r \rho^{\eps}.
\end{equation}

Our aim is to pass to the limit in \eqref{eq:mainHJeps}. To make the convergence result appear naturally, we shall start by providing formal arguments. Assuming Lipschitz bounds on $\varphi^\eps$, and since $\rho^\eps$ is uniformly bounded, the boundedness of $\int_{V} M(v') ( 1- \exp((\varphi^\eps(v) - \varphi^\eps(v'))/\eps) dv'$ implies that we expect the limit $\varphi^0$ to be independent of $v$. To identify the limit $\varphi^0$, we shall thus perform the following expansion
\begin{equation}\label{eq:WKBansatz}
\varphi^{\eps}(t,x,v)=\varphi^0(t,x)+\eps\eta(t,x,v).
\end{equation}
Plugging the latter into \eqref{eq:WKBansatz} yields
\begin{equation*}
\partial_t \varphi^0 + v\cdot\nabla_x \varphi^0+r=(1+r)\int_V M(v')\left(1-e^{\eta(v)-\eta(v')}\right)dv' + re^{-\frac{\varphi^0}{\eps}}\int_V e^{-\eta(v')}dv'.
\end{equation*}
%
As a consequence, for any $(t,x) \in \left\{ \varphi^0>0\right\}$, we have
\begin{equation}\label{eq:spectralpb1}
\partial_t \varphi^0 + v\cdot\nabla_x \varphi^0=1-e^{\eta(v)}(1+r)\int_V M(v')e^{-\eta(v')}dv'.
\end{equation}
One should read this equation as an eigenvalue problem in the velocity variable. Indeed, setting 
\begin{equation*}
p(t,x) = \nabla_x \varphi^0(t,x), \qquad \eta(t,x,v) := - \ln \left(\frac{Q_{p(t,x)}}{M(v)}\right), \qquad H(p(t,x)):=-\partial_t \varphi^0(t,x),
\end{equation*}
we see that $(H,Q)$ are the principal eigenelements of the following spectral problem
\begin{equation*}
(1+r)M(v) \int_V Q_p(v') \, dv' -  \left( 1 - v \cdot p  \right) Q_p(v) = H(p) Q_p(v).
\end{equation*}
The dependency with respect to $r$ can be identified by setting $p':=\frac{p}{1+r}$, $\mathcal{H}(\cdot):=\frac{H((1+r)\cdot)-r}{1+r}$ and $\widetilde{Q}_{p'}=Q_p$. Indeed, we have then that $\partial_t \varphi^0 + (r+1)\mathcal{H}(\frac{p}{r+1})+r = 0$ and the Hamiltonian $\mathcal{H}$ is given by 
\begin{equation}\label{eq:spectralproblem}
\left( 1+ \mathcal{H}\left(p'\right) - v \cdot p'  \right) \widetilde{Q}_{p'}(v) = M(v) \int_V \widetilde{Q}_{p'}(v') \, dv'.
\end{equation}

After these heuristics, we are now ready to define properly the Hamiltonian $\mathcal{H}$ involved. 
\begin{definition}\label{def:hamiltonian}
We define, for $e \in \mathbb{S}^{n-1}$,
\begin{equation*}
l(e) = \int_V \frac{M(v)}{\overline{v}(e)-v \cdot e}dv.
\end{equation*}
The so-called singular set is defined by
\begin{equation}\label{eq:Sing}
\mathrm{Sing}\left(M\right):=\left\{p\in\mathbb{R}^n,  \int_V \frac{M(v)}{\mu(p)-v\cdot p}dv\leq 1\right\} =\left\{p\in\mathbb{R}^n, \, l\left( \frac{p}{\vert p \vert } \right) \leq \vert p \vert \right\}.
\end{equation}
Then, the Hamiltonian $\mathcal{H}$ involved in this paper is given as follows:
\begin{itemize}
\item If $p\notin \mathrm{Sing}\left(M\right)$, then $\mathcal{H}$ is uniquely defined by the following implicit relation :
\begin{equation}\label{eq:defHimplicit}
\int_V \frac{M(v)}{1+ \mathcal{H}(p) -v\cdot p}dv=1,
\end{equation}
\item else, $\mathcal{H}(p)=  \mu\left(p \right)- 1 $.
\end{itemize}
\end{definition}

The relevancy of such a definition, \textit{i.e.} the resolution of \eqref{eq:spectralproblem}, will be discussed in \Cref{sec:HJ} below. With this definition in hand, the convergence result for the sequence of functions $\varphi^\eps$ is as follows.

\begin{theorem}\label{thm:HJlimit}
Suppose that (H0) and (H1) hold, and that the initial data satisfies 
\begin{equation*}
\forall (x,v) \in \R^n \times V, \qquad \varphi^{\eps}(0,x,v) = \varphi_0(x,v).
\end{equation*}
Then, $\left( \varphi^{\eps} \right)_\eps$ converges uniformly on all compacts of $\mathbb{R}_+^*\times \mathbb{R}^n\times V$ towards $\varphi^0$, where $\varphi^0$ does not depend on $v$. Moreover $\varphi^0$ is the unique viscosity solution of the following Hamilton-Jacobi equation:
\begin{equation}\label{eq:varHJ}
\begin{cases}
\min\left \lbrace \partial_t \varphi^0 + (1+r) \mathcal{H} \left(\frac{\nabla_x \varphi^0}{1+r} \right) + r , \varphi^0 \right\rbrace = 0, & \qquad  (t,x) \in \R_+^* \times \R^n, \medskip \\
\varphi^0(0,x)= \underset{v\in V}{\mathrm{min}}\,\varphi_0(x,v),& \qquad x \in \R^n.
\end{cases}
\end{equation}
\end{theorem}

Let us now emphasize the differences between the result presented here and the very related works \cite{bouin_kinetic_2012,bouin_hamilton-jacobi_2015,caillerie_large_2017}. First, the results from \cite{bouin_kinetic_2012} and \cite{bouin_hamilton-jacobi_2015} only hold for $n=1$ and for $M\geq \delta>0$. In \cite{bouin_hamilton-jacobi_2015}, the first author successfully proved a convergence result in the case $r>0$. It is worth mentioning that a much wider class of collision operators was considered in \cite{bouin_hamilton-jacobi_2015}, but under the condition of existence of a $L^1$ eigenvector. We believe that the ideas of the present work could be used there, but with technicalities inherent from the spectral problem that would require a special study.

As explained before, the multidimensional case ($n>1$) is more delicate since the relation \eqref{eq:defHimplicit} may not have a solution. We refer to our \Cref{ex:counterexample} for a situation where this happens. In \cite{caillerie_large_2017}, the second author generalized the convergence result of \cite{bouin_kinetic_2012} in the multidimensional case, with no reaction term. However, the proof we design in this paper is simpler and more adaptable.  For this we manage to use the half-relaxed limits of Barles and Perthame \cite{barles_exit_1988} in the spirit of \cite{bouin_hamiltonjacobi_2015}. 
%
%
We point out that an asymptotic preserving scheme has been developed by Hivert in \cite{hivert_asymptotic_2017} to numerically solve \eqref{eq:mainHJeps} using the Hamilton-Jacobi framework developed in \cite{bouin_hamilton-jacobi_2015}.
We present the proof of \Cref{thm:HJlimit} in \Cref{sec:HJ} below.

\subsection*{Travelling wave solutions and spreading of planar like initial data}

We then investigate the existence of travelling wave solutions of \eqref{eq:main}. As in the mono-dimensional case treated in \cite{bouin_propagation_2015}, we will prove that there exists a minimal speed $c^*$ for which travelling wave solutions exist. We will use the following definition throughout the paper. 

\begin{defi}\label{def:deftw}
A function $f$ is a travelling wave solution of speed $c \in \R_+$ and direction $e\in\mathbb{S}^{n-1}$ of equation \eqref{eq:main} if it can be written $f(t,x,v) = \tilde f \left( x \cdot e - ct , v \right)$, where  the profile  \NC{$\tilde f \in \mathcal{C}^2 \left( \R,  L^1(V) \right)$} solves 
\begin{equation}
\left( v \cdot e - c \right) \partial_\xi \tilde f = M(v) \tilde\rho - \tilde f + r \tilde\rho  \left( M(v) - \tilde f\right) 
\end{equation}
and satisfies 
\begin{equation} \label{eq:deftw}\forall (z,v) \in\R\times V\,, \quad  0\leq \tilde f(z,v)\leq M(v)\,, \quad \lim_{z\to -\infty} \tilde f(z,v)=M(v)\,, \quad \lim_{z\to +\infty} \tilde f(z,v)=0\;.
\end{equation}
\end{defi}

It is well known for this kind of Fisher-KPP type problems that propagation fronts are so-called pulled fronts, that is the speed of propagation is given by seeking exponentially decaying solutions of the linearized problem in a moving frame. As a consequence, for any $\lambda > 0$, one can define $c(\lambda,e)$ using the spectral problem solved in \Cref{def:hamiltonian}. Indeed, we set 
\begin{equation}\label{eq:linspeed}
c(\lambda,e) = (1+r)\mathcal{H} \left(\frac{\lambda e}{1+r} \right) + r.
\end{equation}
Then we have the formula for the minimal speed in the direction $e \in \mathbb{S}^{n-1}$.
\[ c^*(e)= \inf_{\lambda>0} c(\lambda,e)\, .  \]
We obtain the following existence result. 
\begin{thm} \label{thm:existence-tw} 
Let $e\in \mathbb{S}^{n-1}$. For all $c \in [ c^*(e) , \overline{v}(e) )$,  there exists a travelling wave solution of \eqref{eq:main} with speed $c$ and direction $e$. Moreover, there exists no positive travelling wave solution of speed $c\in [0,c^*(e))$.
\end{thm}
%

Following very closely the proof used in the mono-dimensional case, we shall prove this Theorem using sub and super-solution and a comparison principle satisfied by \eqref{eq:main}. We shall construct these sub- and super-solution using travelling wave solutions of the linearized problem. The main difference concerning the travelling wave result is the way we prove the minimality of the speed $c^*(e)$. Indeed, it might happen that $c(\lambda,e)$ is singular at its minimum $\lambda^*$ so that one can not reproduce the same argument as for the mono-dimensional case used in \cite{bouin_propagation_2015}, that was based on the Rouch\'e Theorem. Using the Hamilton-Jacobi framework above, in a similar fashion as in \cite{berestycki_spreading_2012 for the Fisher-KPP equation in an heterogeneous media}, we prove the following result.

\begin{prop} \label{prop:spreadingplanar}
Let $f_0$ be a non-zero initial data, compactly supported in some direction $e_0$, such that there exists $\gamma < 1$ such that 
\begin{equation*}
\gamma M(v) \textbf{1}_{[-x_m,x_m]\cdot e_0 + e_0^\bot}(x) \leq f_0 (x,v)\leq M(v) \textbf{1}_{[-x_M,x_M]\cdot e_0 + e_0^\bot}(x),
\end{equation*}
for all $(x,v) \in \R^n \times V$. Let $f$ be the solution of the Cauchy problem \eqref{eq:main} associated to this initial data. Then we have 
\begin{equation}\label{eq:frontplanar}
\lim_{t \to +\infty} \sup_{x\cdot e_0 > ct} \rho(t,x) = 0,\text{ if } c > c^*(e_0),
\end{equation}
\begin{equation}\label{eq:backplanar}
\lim_{t\to +\infty} f(t,e_0^\bot + c t e_0,v) = M(v) \,,\text{ if } c < c^*(e_0),
\end{equation}
uniformly in $v \in V$.
\end{prop}

\subsection*{Spreading of compactly supported initial data}

Finally, we also deduce from the Hamilton-Jacobi framework a spreading result for initial conditions that are compactly supported. To this aim, let us first define the speed $w^*(e_0)$ associated to any direction $e_0 \in \mathbb{S}^{n-1}$ via the following Freidlin-G\"artner formula (see \cite{freidlin_propagation_1979} for its first derivation). 
\begin{equation*}
w^*(e_0) = \min_{\substack{e \in \mathbb{S}^{n-1}\\e_0 \cdot e > 0}}\left( \frac{c^*(e)}{e_0 \cdot e}\right).
\end{equation*}
We obtain the following result.

\begin{prop} \label{prop:spreadingbounded}
Let $f^0$ be a non-zero compactly supported initial data such that $0\leq f_0 (x,v)\leq M(v)$ for all $(x,v) \in \R^n \times V$. Let $f$ be the solution of the Cauchy 
problem \eqref{eq:main} associated to this initial data. Then for any $e_0 \in \mathbb{S}^{n-1}$ and all $x\in \R^n$, we have
\begin{equation}\label{eq:frontcompact}
\lim_{t \to \infty} f(t,x + c t e_0,v) = 0, \qquad \text{ if } c > w^*(e_0),
\end{equation}
pointwise
and
\begin{equation}\label{eq:backcompact}
\lim_{t \to \infty} f(t,c t e_0,v)  = M(v), \qquad \text{ if } 0 \leq c < w^*(e_0),
\end{equation}
for all $v \in V$.
\end{prop}

This result is interesting since contrary to the case of the usual Fisher-KPP equation in heterogeneous domains, where the Freidlin-G\"artner formula holds, see \cite{rossi_freidlin-gartner_2015}, here there is no heterogeneity in space. The heterogeneity coming potentially from the velocity set, this would not be present in the macroscopic limit (the Fisher-KPP equation, see above). Of course, if $V$ is rotationally symmetric, the speed $w^*$ is independent of the direction, and the propagation is radial.

One could wonder how the shape of $V$ (\textit{e.g.} its topological properties or its convexity) influences the shape of the front and the speed of propagation. We will investigate this question in a future work.


The rest of this paper is organized as follows. In \Cref{sec:HJ}, we prove the Hamilton-Jacobi limit. We discuss the construction of travelling waves and the spreading results in \Cref{sec:TW}.

\subsection*{Acknowledgements}

This work has started after a discussion with Olivier Benichou and Raphaël Voituriez. The authors thank warmly Vincent Calvez for early discussions about this problem, and for a careful reading of the manuscript. This work has benefited of an invaluable insight from Guy Barles, which, while originally meant for another manuscript, found a direct application in the present one. The authors are more than deeply grateful to Guy Barles for this. EB acknowledges the green hospitality of the University of Cambridge during the second semester of the academic year 2015-2016. EB and NC acknowledge the support of the ERC Grant MATKIT
(ERC-2011-StG). NC thanks the University of Cambridge for its hospitality during his three week stay over Spring 2016.  In addition, NC has received funding from the European Research Council (ERC) under the European Union's Horizon 2020 research and 
innovation program (grant agreement No 639638). EB dedicates this work to Paul-Antoine Girard, as an encouragement to never give up and always slide towards his highest ambitions.

\section{The Hamilton-Jacobi limit}\label{sec:HJ}

In this Section, we present the proof of the convergence result \Cref{thm:HJlimit}. We then prove a convergence result for $\rho^\eps$ in the region $\lbrace \varphi^0 = 0 \rbrace$. This result will help us to show that the speed of propagation is still the minimal speed of existence of travelling waves, despite the singularity of $\mathcal{H}$.

\subsection{The spectral problem}\label{thespectralproblem}

In this Section, we discuss the resolution of the spectral problem given by \eqref{eq:spectralproblem}. We also provide examples for which the singular set of $M$ is not empty.

\subsubsection{The resolution}

For any $p'>0$, we look for an eigenvalue $\mathcal{H}(p')$ associated to a positive eigenvector $\widetilde{Q}_{p}$ such that 
\begin{equation*}
\left( 1 + \mathcal{H}(p') - v \cdot p' \right) \widetilde{Q}_{p'}(v) = M(v)\int_V \widetilde{Q}_{p'}(v') \, dv', \qquad v \in V.
\end{equation*}
Note that it may happen that $\widetilde{Q}_{p'}$ has a singular part. Since the problem is linear, one can always assume that $\widetilde{Q}_{p'}$ is a probability measure. We are thus led to find an eigenvalue $\mathcal{H}(p')$ such that there exists a probability measure $\widetilde{Q}_{p'}$\footnote{To avoid too many notation, we identify $\widetilde{Q}_{p'}$ to its density when relevant.} such that 
\begin{equation*}
\left( 1 + \mathcal{H}(p') - v \cdot p' \right) \widetilde{Q}_{p'}(v) = M(v), \qquad v \in V.
\end{equation*}

To make the singular set $\mathrm{Sing}\left(M\right)$ appear naturally, let us first  investigate the case when $\widetilde{Q}_{p'} \in L^1(V)$. If a solution exists, then the profile $\widetilde{Q}_{p'}$ necessarily satisfies the following equation:
\begin{equation}\label{eq:profile}
\widetilde{Q}_{p'}(v) = \frac{M(v)}{ 1 + \mathcal{H}(p') - v \cdot p'}, \qquad v \in V.
\end{equation}
This is only possible if such an expression defines a probability measure. As a consequence, one shall look for conditions under which there exists $\mathcal{H}(p')$ such that 
\begin{equation*}
I \left(\mathcal{H}(p'),p' \right) := \int_V \frac{M(v')}{1 + \mathcal{H}(p') - v' \cdot p'} dv' = 1,
\end{equation*}
with $1 + \mathcal{H}(p') - v' \cdot p > 0$ for all $v' \in V$, that is $\mathcal{H}(p') > \mu(p') - 1$.

For any $p' \notin \text{Sing}(M)$, since the function $\xi \mapsto I(\xi,p')$ is decreasing over $(\mu(p') - 1,+\infty)$, $\mathcal{H}(p')$ exists and is unique in this interval since $I(\mu(p') - 1,p') > 1$, by the definition of $p'$ not being in the singular set. 

However, for any $p' \in \text{Sing}(M)$, it is not possible to solve $I \left(\mathcal{H}(p'),p' \right) = 1$ since $I(\mu(p') - 1,p') \leq 1$. After Theorem 1.2 in \cite{coville_singular_2013-1}, there exists a solution to \eqref{eq:spectralproblem}, given by the couple $(\mathcal{H}(p'),\widetilde{Q}_{p'})$ where $\mathcal{H}(p')=\mu(p')-1$ and $\widetilde{Q}_{p'}$ is a positive measure given by:
\begin{equation*}
\widetilde{Q}_{p'} :=\frac{M(v)}{\mu(p')-v\cdot p'} dv + \left( 1 - \int_V \frac{M(v)}{\mu(p')-v\cdot p'} dv \right) \delta_w,
\end{equation*}
 where $\delta_w$ is the dirac mass located at $w\in \mathrm{Arg}\,\mu(p')$.
 
From \cite{caillerie_large_2017}, we know that the set $\mathrm{Sing}(M)^{c}$ is convex and contains $0$. To identify the different cases where such a singularity set may occur, we detail three examples hereafter.
 

\subsubsection{Examples}

\begin{example}\label{ex:nosing}
In the one-dimensional case ($n=1$), we have $\mathrm{Sing}(M)=\emptyset$ when $\underset{v\in V}{\mathrm{inf}}\,M(v)>0$ since
\begin{equation}\label{eq:ex}
\int_{-\overline{v}}^{\overline{v}}\frac{M(v')}{\mu(p')-vp'}dv= \int_{-\overline{v}}^{\overline{v}}\frac{M(v)}{|p'|v_{max}-vp'}dv\geq \frac{\mathrm{inf}\,M(v)}{|p'|\overline{v}}\cdot \int_{-1}^{1}\frac{dv}{1-v}=+\infty.
\end{equation}
By monotone convergence we have
\begin{equation*}
\underset{H\to \mu(p')-1}{\mathrm{lim}}\,\int_{-\overline{v}}^{\overline{v}}\frac{M(v)}{1+H -vp'}dv = +\infty,
\end{equation*}
hence, for all $p'\in\mathbb{R}$, there exists a unique $\mathcal{H}(p')$ that solves the spectral problem in $L^1(V)$.
\end{example}

This latter framework is the one used in \cite{bouin_propagation_2015}. In fact, we can only require that $M$ does not cancel in a neighborhood of $v=\overline{v}$ in order to get $\mathrm{Sing}(M)=\emptyset$. Indeed, the integral in \eqref{eq:ex} will also diverge in that scenario. If $M(\overline{v})=0$, this argument may not work out. Consider for example:
\begin{example}\label{ex:annulebord}
Let $n=1$, $V=[-1,1]$ and $M(v)=\frac{3}{2}(1-|v|)^2$. Then,
\begin{align*}
l(1)=\int_{-1}^{1}\frac{M(v)}{1-v} \,dv &=\frac{3}{2}\int_{-1}^{1}\frac{(1-|v|)^2}{1-v} \, dv = 3 \int_{0}^{1}\frac{(1-v)^2}{(1-v)(1+v)} \, dv \\
&= 3 \int_{0}^{1}\frac{1-v}{1+v} \, dv = 3 \int_{0}^{1}\frac{2-(1+v)}{1+v} \, dv\\ 
&= 3 \left( \int_{0}^{1}\frac{2}{1+v} \, dv - 1 \right) = 3(2\ln(2)-1).
\end{align*}
Hence, $|p'| \geq 3(2\ln(2)-1)$ if and only if
\begin{equation*}
\int_{-1}^{1}\frac{M(v)}{\mu(p')-vp'}dv=\frac{3}{2|p'|}\int_{-1}^{1}\frac{(1-|v|)^2}{1-v}dv\leq 1,
\end{equation*}
therefore, $\mathrm{Sing}(M)=\left(-3(2\ln(2)-1),3(2\ln(2)-1)\right)^c$. Let us also notice that
\begin{align*}
\int_V \frac{M(v)}{(\overline{v}(1)-v)^2}dv &=\frac{3}{2}\int_{-1}^{1}\frac{(1-|v|)^2}{(1-v)^2}dv=3\int_{0}^{1}\frac{1+v^2}{(1+v)^2}dv\\
&=3\int_{0}^{1}\frac{(1+v)^2 - 2(1+v)+2}{(1+v)^2}dv\\
&=3\int_{0}^{1} \left( 1 - \frac{2}{1+v}+ \frac{2}{(1+v)^2} \right)dv\\
&=3(1 -2 \ln(2) + 1) = 6(1 - \ln(2)) <+\infty.
%
\end{align*}
We will make a use of this result later.
\end{example}

In the multi-dimensional case, we may encounter a singular set, even when $\underset{v\in V}{\mathrm{inf}}\,M(v)>0$. These singularities can occur in the simplest cases.

\begin{example}\label{ex:counterexample}
Let $n\geq1$, let $V=B(0,1)$ be the $n$-dimensional unit ball. Let $e=e_1$ and $M=\omega_{n}^{-1}.\mathds{1}_{\overline{B\left(0,1\right)}}$,
where $\omega_{n}$ is the Lebesgue measure of $V$. For $n=1$, since $M>0$ we have $\mathrm{Sing}(M)=\emptyset$ (recall \cref{ex:nosing}). Suppose now that $n>1$. Then, 
\begin{eqnarray*}
 l(e_1) &=&\int_{B\left(0,1\right)}\frac{M(v)}{\overline{v}(e_1)-v\cdot e_1}dv \\
 & = & \frac{1}{\omega_{n}}\int_{B\left(0,1\right)}\frac{1}{1-v_1}dv\\
&=&\frac{1}{\omega_{n}}\int_{-1}^{1}\frac{1}{1-v_{1}}\left(\int \mathds{1}_{\left\{ v_1^2+v_2^2+\ldots+v_n^2\leq1\right\}}(v_2,\ldots,v_n)dv_2\ldots dv_n\right)dv_1.\\
\end{eqnarray*}
Now, for fixed $v_1$, the quantity $\int \mathds{1}_{\left\{ v_1^2+v_2^2+\ldots+v_n^2\leq1\right\}}(v_2,\ldots,v_n)dv_2\ldots dv_n$ is the Lebesgue measure of the $(n-1)$-dimensional ball of radius $\sqrt{1-v_1^2}$, hence
\begin{equation*}
\int \mathds{1}_{\left\{ v_1^2+v_2^2+\ldots+v_n^2\leq1\right\}}(v_2,\ldots,v_n)dv_2\ldots dv_n=\omega_{n-1}\times \left(\sqrt{1-v_1^2}\right)^{n-1}.
\end{equation*}
Finally,
\begin{align*}
l(e_1)& = \frac{\omega_{n-1}}{\omega_{n}}\int_{-1}^{1} \frac{\left(1-v_{1}^{2}\right)^{\frac{n-1}{2}}}{1-v_1} dv_{1} \\
&=\frac{2\omega_{n-1}}{\omega_{n}}\int_{0}^{1} \left(1-v_{1}^{2}\right)^{\frac{n-3}{2}} dv_{1}\\
&=\frac{2\omega_{n-1}}{\omega_{n}}\int_{0}^{\frac{\pi}{2}} \left(\cos(\theta)\right)^{n-2} d\theta,\\
&=  \frac{n}{n-1},
\end{align*}
where we have used, for example, the relationship between the volume of the unit ball and the Wallis integrals. By rotational invariance, $\mathrm{Sing}(M)=B\left(0,\frac{n-1}{n}\right)^{c}$.
\end{example}

\subsection{Proof of \Cref{thm:HJlimit}}

In this Section, we now prove \Cref{thm:HJlimit}. We will use the half-relaxed limits method of Barles and Perthame \cite{barles_exit_1988}. In addition to that, and similarly to the papers \cite{bouin_kinetic_2012,caillerie_large_2017}, we need to use the perturbed test-function method. We emphasize that the corrected test function is defined thanks to the spectral problem \eqref{eq:spectralproblem}, keeping only the regular part of the eigenfunction (recall that it may have singularities).

Since the sequence $\varphi^\eps$ is uniformly bounded by the maximum principle (check Proposition 5 in \cite{bouin_hamilton-jacobi_2015}), we can define its upper- and lower- semi continuous envelopes by the following formulas
\begin{equation*}
\varphi^*(t,x,v) = \limsup_{\substack{\eps \to 0\\(s,y,w) \to (t,x,v)}} \varphi^\eps(s,y,w), \qquad \varphi_*(t,x,v) = \liminf_{\substack{\eps \to 0\\(s,y,w) \to (t,x,v)}} \varphi^\eps(s,y,w)
\end{equation*}
Recall that $\varphi^*$ is upper semi-continuous, $\varphi_*$ is lower semi-continuous and that from their definition, one has $\varphi_* \leq \varphi^*$. We have the following: 
\begin{prop}
Let $\varphi^\eps$ be a solution to \eqref{eq:mainHJeps}. 
\begin{enumerate}\label{prop:semilimits}
\item[(i)] The upper semi-limit $\varphi^*$ is constant with respect to the velocity variable on $\R_+^* \times \R^n$.
\item[(ii)] The function $(t,x) \mapsto \varphi^*(t,x)$ is a viscosity sub-solution to \eqref{eq:varHJ} on $\R_+^* \times \R^n$.
\item[(iii)] The function $(t,x) \mapsto \min_{w \in V} \varphi_*(t,x,w)$ is a viscosity super-solution to \eqref{eq:varHJ} on $\R_+^* \times \R^n$.
\end{enumerate}
\end{prop}

We recall that for all $(t,x)$, the minimum $\min_{w \in V} \varphi_*(t,x,w)$ is attained since $V$ is bounded and $\varphi_*$ is lower semi-continuous. 
We point out here that if $r=0$, that is, the case of \cite{bouin_kinetic_2012,caillerie_large_2017}, it is not necessary to prove that $\varphi^*$ is constant in the velocity variable. One can replace this by proving that $\max_{w \in V} \varphi^*$ is a sub-solution to \eqref{eq:varHJ}. The fact that $\varphi^*$ is constant in the velocity variable is needed to control the limit of $\rho^\eps$.

\begin{proof}[{\bf Proof of \Cref{prop:semilimits}}]

We start with the proof of (i). Take $(t^0,x^0,v^0) \in \R^{*}_+ \times \R^n \times V$. Let $\psi$ be a test function such that $\varphi^* - \psi$ has a strict local maximum at $(t^0,x^0,v^0)$. Then there exists a sequence $\left( t^{\eps},x^{\eps},v^{\eps}\right)$ of maximum points of $\varphi^{\eps}-\psi$ satisfying $(t^{\eps},x^{\eps},v^{\eps})\to(t^0,x^0,v^0)$. From this we deduce that $\lim_{\eps \to 0} \varphi^{\eps}(t^{\eps},x^{\eps},v^{\eps}) = \varphi^*(t,x,v)$. Recalling \eqref{eq:mainHJeps}, we have at $\left( t^{\eps},x^{\eps},v^{\eps}\right)$:
\begin{equation*}
\partial_t \psi + v^{\eps} \cdot \nabla_x \psi +r = (1+r)\left(1-\int_V M(v') e^{\frac{\varphi^{\eps} -\varphi^{\eps'}}{\eps}} dv'\right) +r\rho^{\eps}.
\end{equation*}
From this, we deduce that 
\begin{equation*}
\int_{V'} M(v') e^{\frac{\varphi^{\eps}(t^{\eps},x^{\eps},v^{\eps}) -\varphi^{\eps}(t^{\eps},x^{\eps},v')}{\eps}} dv'
\end{equation*}
is uniformly bounded for any $V' \subset V$. By the Jensen inequality, 
\begin{equation*}
\exp\left( \frac{1}{\eps\left\vert V'\right\vert_M} \int_{V'} \left(\varphi^{\eps}(t^{\eps},x^{\eps},v^{\eps}) -\varphi^{\eps}(t^{\eps},x^{\eps},v') \right) M(v') dv'  \right) \leq \frac{1}{\left\vert V' \right\vert_M}\int_{V'} M(v') e^{\frac{\varphi^{\eps}(t^{\eps},x^{\eps},v^{\eps}) -\varphi^{\eps}(t^{\eps},x^{\eps},v')}{\eps}} dv',
\end{equation*}
where $\left\vert V' \right\vert_M:=\int_{V'}M(v)dv$. Thus, 
\begin{equation*}
\limsup_{\eps \to 0} \left( \int_{V'} \left(\varphi^{\eps}(t^{\eps},x^{\eps},v^{\eps}) -\varphi^{\eps}(t^{\eps},x^{\eps},v') \right) M(v') dv' \right) \leq 0.
\end{equation*}
We write
\begin{align*}
\int_{V'} \left(\varphi^{\eps}(v^{\eps}) -\varphi^{\eps}(v') \right) M(v') dv' &= \int_{V'} \left[\left( \varphi^{\eps}(v^{\eps}) - \psi(v^{\eps})\right) - \left( \varphi^{\eps}(v') - \psi(v') \right) + \left( \psi(v^{\eps}) - \psi(v')\right)\right] M(v') dv'\\
&= \int_{V'} \left[\left( \varphi^{\eps}(v^{\eps}) - \psi(v^{\eps})\right) - \left( \varphi^{\eps}(v') - \psi(v') \right) \right] M(v') dv'\\ &\hfill + \int_{V'} \left( \psi(v^{\eps}) - \psi(v')\right) M(v') dv'
\end{align*}
We can thus use the Fatou Lemma, together with $- \limsup_{\eps \to 0} \varphi^{\eps}(t^{\eps},x^{\eps},v') \geq - \varphi^*(t^0,x^0,v')$ to get
\begin{align*}
\left(\int_{V'} M(v')dv'\right) \varphi^*(v^0) - \int_{V'}\varphi^*(v')M(v') dv' & = \int_{V'} \left(\varphi^*(v^0) -\varphi^*(v') \right) M(v') dv' \\
&\leq
\int_{V'} \liminf_{\eps \to 0}  \left(\varphi^{\eps}(v^{\eps}) -\varphi^{\eps}(v') \right) M(v') dv' \\
&\leq \liminf_{\eps \to 0} \left( \int_{V'} \left(\varphi^{\eps}(v^{\eps}) -\varphi^{\eps}(v') \right) M(v') dv' \right)\\ 
&\leq \limsup_{\eps \to 0} \left( \int_{V'} \left(\varphi^{\eps}(v^{\eps}) -\varphi^{\eps}(v') \right) M(v') dv' \right)\\
&\leq 0, 
\end{align*}
We shall deduce, since the latter is true for any $\vert V' \vert$ that
\begin{equation*}
\varphi^*(t^0,x^0,v^0) \leq \inf_V \varphi^*(t^0,x^0,\cdot)
\end{equation*}
and thus $\varphi^*$ is constant in velocity. 

We now continue with the proof of (ii). We have to prove that on $\lbrace\varphi^* >0\rbrace\cap(\R_+^* \times \R^n)$, the function $\varphi^*$ is a viscosity subsolution of \eqref{eq:varHJ}. To this aim, let $\psi \in C^2(\R_+^* \times \R^n)$ be a test function such that $\varphi^*-\psi$ has a local maximum in $(t^0,x^0)\in (\R_+^* \times \R^n) \cap \left\{\varphi^*>0\right\}$. We denote by $p^0(t^0,x^0) = \frac{\nabla_x \psi(t^0,x^0)}{1+r}$.

\medskip

{\bf \# First case : $p^0(t^0,x^0) \notin\mathrm{Sing}\,M$.} 

\medskip

\noindent We define a corrector $\eta$ according to the following formula: 
\begin{equation*}
\eta(v) = \ln \left( 1 + \mathcal{H}\left(p^0(t^0,x^0)\right) - v \cdot p^0(t^0,x^0) \right) 
\end{equation*}
Let us define the perturbed test function $\psi^{\eps}:=\psi +\eps\eta$. We recall the fact that in this case $\int_V M(v') \exp(-\eta(v')) dv' = 1$. The function $\psi^\eps$ converges uniformly to $\psi$ since $\eta$ is bounded on $V$. As a consequence, there exists a sequence $\left( t^{\eps},x^{\eps},v^{\eps}\right)$ of maximum points of $\varphi^{\eps}-\psi^{\eps}$ satisfying $(t^{\eps},x^{\eps})\to(t^0,x^0)$ and such that $\lim_{\eps \to 0} \varphi^{\eps}(t^\eps,x^\eps,v^\eps) = \varphi^*(t^0,x^0)$. Recalling \eqref{eq:mainHJeps}, we have at $\left( t^{\eps},x^{\eps},v^{\eps}\right)$:
\begin{equation*}
\partial_t \psi^{\eps} + v^{\eps} \cdot \nabla_x \psi^{\eps}+r = (1+r)\left(1-\int_V M(v') e^{\frac{\varphi^{\eps} -\varphi^{\eps'}}{\eps}} dv'\right) +r\rho^{\eps}.
\end{equation*}
Since $\left( t^{\eps},x^{\eps},v^{\eps}\right)$ is a maximum point, we may rearrange the r.h.s. of the latter so that the previous equation may be rewritten as follows
\begin{eqnarray*}
\partial_t \psi^{\eps} + v^{\eps} \cdot \nabla_x \psi^{\eps} +r & \leq & (1+r)\left(1-\int_V M(v') \exp \left(\eta(v^\eps)-\eta(v')\right) dv'\right)+r\rho^{\eps},\\
 & = & (1+r)\left(1- \left( \int_V M(v') \exp\left(-\eta(v')\right) dv'\right) \exp \left(\eta(v^\eps)\right) \right)+r\rho^{\eps},\\
 & = & (1+r)\left(1- \exp \left(\eta(v^\eps)\right) \right)+r\rho^{\eps},\\
& = &  - (1+r) \mathcal{H}\left(p^0(t^0,x^0)\right) + v^\eps \cdot \nabla_x \psi^0(t^0,x^0) + r\rho^\eps.
\end{eqnarray*}
Since $(t^0,x^0)\in\left\{ \varphi^*>0\right\}$ and $\lim_{\eps \to 0} \varphi^{\eps}(t^\eps,x^\eps,v^\eps) = \varphi^*(t^0,x^0)$, we have that, eventually, $\varphi^{\eps}(t^\eps,x^\eps,v^\eps) > \varphi^*(t^0,x^0)/2 >0$ for $\eps$ sufficiently small. Since 
\begin{equation*}
r\rho^\eps \left( e^\frac{\varphi^{\eps}}{\eps} - 1 \right) = \left(1-\int_V M(v') e^{\frac{\varphi^{\eps} -\varphi^{\eps'}}{\eps}} dv'\right)- \left( \partial_t \psi^{\eps} + v^{\eps} \cdot \nabla_x \psi^{\eps} \right),
\end{equation*}
and the latter r.h.s. is uniformly bounded from above in $\eps$, we deduce that $\lim_{\eps \to 0}\rho^\eps(t^\eps,x^\eps) = 0$. Taking the limit $\eps\to 0$, we get 
%
%
%
\begin{eqnarray*}
\partial_t \psi \left( t^{0},x^{0}\right) + (1+r)\mathcal{H}\left(\frac{\nabla_x \psi (t^0,x^0)}{1+r}\right)+r \leq 0.
\end{eqnarray*}

\medskip

{\bf \# Second case : $p^0(t^0,x^0) \in \mathrm{Sing}\,M$.} 

\medskip

\noindent Let $v^*\in \mathrm{Arg} \, \mu(p^0(t^0,x^0))$. The function $(t,x)\mapsto \phi^\eps (t,x,v^*)-\psi(t,x)$ has a local maximum at a point $(t^\eps,x^\eps)$ satisfying $(t^\eps,x^\eps)\to(t^0,x^0)$ as $\eps\to0$. We then have:
\begin{eqnarray*}
\partial_t \psi (t^\eps,x^\eps)+ v^* \cdot \nabla_x \psi(t^\eps,x^\eps) + r & = & \partial_t \phi^\eps (t^\eps,x^\eps,v^*)+v^* \cdot \nabla_x \phi^{\eps}(t^\eps,x^\eps,v^*) + r \\
 & = & (1+r)\int_{V} M(v') \left( 1-e^{\frac{\varphi^\eps(v^*) - \varphi^\eps(v')}{\eps}} \right) dv'  +  r \rho^{\eps} \\
 & \leq & (1+r) +  r \rho^{\eps}.
\end{eqnarray*}
Since $(t^0,x^0)\in\left\{ \varphi^*>0\right\}$, we have $\rho^{\eps}(t^{\eps},x^{\eps})\to 0$. As a consequence, taking the limit $\eps\to0$, we get 
\begin{equation*}
\partial_t \psi(t^0,x^0) +\mu(\nabla_x \psi(t^0,x^0))  \leq 1.
\end{equation*}

We finally turn to the proof of (iii). That is, the fact that on $\R_+^* \times \R^n$, the function $\min_{w \in V} \varphi_*(\cdot,w)$ is a viscosity supersolution of \eqref{eq:varHJ}. 

Let $\psi \in C^1(\R_+^* \times \R^n)$ be a test function such that $\min_{w \in V} \varphi_* - \psi$ has a local minimum in $(t^0,x^0)\in \R_+^*$. We denote by $p^0(t^0,x^0) = \frac{\nabla_x \psi(t^0,x^0)}{1+r}$. We define the truncated corrector $\eta_\delta$,
\begin{eqnarray*}
\eta(v) & = & \ln \left( 1 + \mathcal{H}\left(p^0(t,x)\right) - v \cdot p^0(t,x) \right),\\
\eta_\delta(v) & = & \max\left( \eta(v) , -1/\delta \right).
\end{eqnarray*}
%
%
%
Let us define the perturbed test function $\psi^{\eps}:=\psi +\eps\eta_\delta$. 
For any $\delta>0$, the function $\psi^\eps$ converges uniformly to $\psi$ as $\eps\to 0$ since $\eta_\delta$ is bounded on $V$. Since $\varphi_*(t^0,x^0,\cdot)$ attains its minimum at, say, $v^0$, we have, for all $v \in V$ and locally in the $(t,x)$ variables, 
\begin{equation*}
\varphi_*(t^0,x^0,v^0) - \psi(t^0,x^0) =\min_{w \in V} \varphi_*(t^0,x^0) - \psi(t^0,x^0) \leq \min_{w \in V} \varphi_*(t,x) - \psi(t,x) \leq \varphi_*(t,x,v) - \psi(t,x),
\end{equation*}
and thus $(t^0,x^0,v^0)$ is a local minimum of $\varphi_* - \psi$, strict in the $(t,x)$ variables. By the definition of the lower semi-limit, there exists a sequence $\left( t_\delta^{\eps},x_\delta^{\eps},v_\delta^{\eps}\right)$ of minimum points of $\varphi^{\eps}-\psi^{\eps}$ satisfying $(t_\delta^{\eps},x_\delta^{\eps})\to(t^0,x^0)$. 
We obtain, after \eqref{eq:mainHJeps}, at the point $\left( t_\delta^{\eps},x_\delta^{\eps},v_\delta^{\eps}\right)$ ,
\begin{eqnarray*}
\partial_t \psi^{\eps} + v_\delta^{\eps} \cdot \nabla_x \psi^{\eps} +r & \geq & (1+r)\left(1-\int_V M(v') \exp \left(\eta_\delta(v^\eps)-\eta_\delta(v')\right) dv'\right),\\
 & = & (1+r)\left(1- \left( \int_V M(v') \exp\left(-\eta_\delta(v')\right) dv'\right) \exp \left(\eta_\delta(v_\delta^\eps)\right) \right).
\end{eqnarray*}
Since the sequence $v_\delta^\eps$ lies in a compact set, taking the limit $\eps\to 0$ (up to extraction), we obtain $v_\delta^0$ such that
\begin{eqnarray*}
\partial_t \psi + v_\delta^{0} \cdot \nabla_x \psi +r & \geq & (1+r)\left(1- \left( \int_V M(v') e^{-\eta_\delta(v')} dv'\right) e^{\eta_\delta(v_\delta^0)} \right).
\end{eqnarray*}
By construction, $\eta_\delta \geq \eta$. As a consequence, $\int_V M(v') e^{-\eta_\delta(v')} dv' \leq \int_V M(v') e^{-\eta(v')} dv' \leq 1$. Thus,  
\begin{eqnarray*}
\partial_t \psi  + v_\delta^{0} \cdot \nabla_x \psi +r & \geq & (1+r)\left(1 - e^{\eta_\delta(v_\delta^0)} \right).
\end{eqnarray*}
We now pass to the limit $\delta \to 0$. By compactness of $V$, one can extract a converging subsequence from $(v_\delta^0)_\delta$, we denote by $v^*$ the limit.
%
%
%

\medskip

{\bf \# First case : $p^0(t^0,x^0) \notin\mathrm{Sing}\,M$.} 

\medskip

\noindent In this case, since $\eta$ is bounded, $\eta_\delta = \eta$ for $\delta$ sufficiently small. Thus, passing to the limit $\delta \to 0$, one gets 
\begin{eqnarray*}
\partial_t \psi + v^* \cdot \nabla_x \psi +r & \geq &(1+r)\left(1- \exp \left(\eta(v^*)\right) \right),\\
& = &  - (1+r) \mathcal{H}\left(p^0(t^0,x^0)\right) + v^* \cdot \nabla_x \psi (t^0,x^0),
\end{eqnarray*}
from which we deduce
\begin{eqnarray*}
\partial_t \psi \left( t^{0},x^{0}\right) + (1+r)\mathcal{H}\left(\frac{\nabla_x \psi (t^0,x^0)}{1+r}\right)+r \geq 0.
\end{eqnarray*}

\medskip

{\bf \# Second case : $p^0(t^0,x^0) \in \mathrm{Sing}\,M$.} 

\medskip

\noindent In this case, the corrector $\eta_\delta$ is
\begin{equation*}
\eta_\delta(v) = \max\left( \mathrm{ln}\left(\mu\left(\nabla_x \psi (t^0,x^0)\right)-v \cdot \nabla_x \psi (t^0,x^0)\right), -1/\delta \right).
\end{equation*}

If $v^* \notin \mathrm{Arg} \, \mu(p^0(t^0,x^0))$, since $\eta$ is bounded on all compacts of $V\setminus \mathrm{Arg} \, \mu(p^0(t^0,x^0))$, $\eta_\delta(v^0_{\delta})=\eta(v^0_{\delta})$ for $\delta$ sufficiently small and we recover the first case.

If $v^*\in \mathrm{Arg} \, \mu(p^0(t^0,x^0))$, then take $\delta' > 0$, one has when $\delta < \delta'$ is sufficiently small,
\begin{equation*}
- \frac{1}{\delta'} = \eta_{\delta'}(v_\delta^0) \geq \eta_{\delta}(v_\delta^0),
\end{equation*}
and thus $\lim_{\delta \to 0} \eta_{\delta}(v_\delta^0) = - \infty$. From that we conclude 
\begin{eqnarray*}
\partial_t \psi + \mu\left(\nabla_x \psi (t^0,x^0)\right) \geq 1.
\end{eqnarray*}
\end{proof}

We now conclude with the proof of the convergence result. For this, we need to input initial conditions. Obviously, one cannot get any uniqueness result for the Hamilton-Jacobi equation \eqref{eq:varHJ} without imposing any initial condition. We now check the initial condition of \eqref{eq:varHJ} in the viscosity sense. 
%
%

\begin{prop}\label{prop:inicond}
If one assumes that $\varphi_0^\eps = \varphi_0$, the sequence $\varphi^\eps$ converges uniformly on compact subsets of $\R_+^* \times \R^n$ to $\varphi^0$, the unique viscosity solution of 
\begin{equation*}
\begin{cases}
\min\left \lbrace \partial_t \varphi^0 + (1+r) \mathcal{H} \left(\frac{\nabla_x \varphi^0}{1+r} \right) + r , \varphi^0 \right\rbrace = 0, & \qquad  (t,x) \in \R_+^* \times \R^n, \medskip \\
\varphi^{0}(0,x) = \underset{v \in V}{\min} \, \varphi_0(x,v) ,& \qquad x \in \R^n.
\end{cases}
\end{equation*}
\end{prop}

\begin{proof}[{\bf Proof of \Cref{prop:inicond}}]
We extend the definition of $\varphi^*$ to $\left\{ t = 0 \right\} \times \R^n$ by the formula
\begin{equation*}
\displaystyle \varphi^*(0,x) = \limsup_{\substack{t \searrow 0^+\\x' \to x}} \varphi^*(t,x').
\end{equation*}
One has to prove the following
\begin{equation}\label{eq:inicondsub}
\min\left(  \min\left \lbrace \partial_t \varphi^* + (1+r) \mathcal{H} \left(\frac{\nabla_x \varphi^*}{1+r} \right) + r , \varphi^* \right\rbrace,   \varphi^* - \min_{v \in V} \varphi_0(\cdot,v)  \right) \leq 0,
\end{equation}
on $\left\{ t = 0 \right\} \times \R^n$ in the viscosity sense. 
%
%
%

Let $\psi \in \mathcal{C}^1 \left( \R^+ \times \R \right)$ be a test function such that $\varphi^*-\psi$ has a strict local maximum at $(t^0 = 0,x^0)$. We now prove that either
\begin{equation*} 
\varphi^*(0,x^0) \leq \min_{v \in V} \varphi_0(x,v),
\end{equation*}
or 
\begin{equation*}
\partial_t \psi + (1+r) \mathcal{H} \left(\frac{\nabla_x \psi}{1+r} \right) + r \leq 0 \qquad \textrm{ when } \qquad \varphi^*(0,x^0) > 0.
\end{equation*}

Suppose then that 
\begin{equation*}
\label{u0x0}
\varphi^*(0,x^0) > \min_{v \in V} \varphi_0(x,v).
\end{equation*}
We shall now prove that 
\begin{equation*}
\partial_t \psi + (1+r) \mathcal{H} \left(\frac{\nabla_x \psi}{1+r} \right) + r \leq 0,
\end{equation*}
since then $\varphi^*(0,x^0) > 0$. We now go through the same steps as for the proof of \Cref{prop:semilimits}, but with slight changes due to the present situation. We keep the same notations. 

\medskip

{\bf \# First case : $p^0(t^0,x^0) \notin\mathrm{Sing}\,M$.} 

\medskip

%
The function $\psi^\eps$ converges uniformly to $\psi$ since $\eta$ is bounded on $V$. Adding this fact to the definition of $\varphi^*(0,x^0)$, we get the existence of a sequence $\left( t^{\eps},x^{\eps},v^{\eps}\right)$ of maximum points of $\varphi^{\eps}-\psi^{\eps}$ satisfying $t^{\eps} > 0$, $(t^{\eps},x^{\eps})\to(0,x^0)$ and such that $\lim_{\eps \to 0} \varphi^{\eps}(t^\eps,x^\eps,v^\eps) = \varphi^*(0,x^0)$. The rest of the proof is similar.

\medskip

{\bf \# Second case : $p^0(t^0,x^0) \in \mathrm{Sing}\,M$.} 

\medskip

\noindent Let $v^*\in \mathrm{Arg} \, \mu(p^0(t^0,x^0))$. As for the previous case, due to the definition of $\varphi^*$, the function $(t,x)\mapsto \phi^\eps (t,x,v^*)-\psi(t,x)$ has a local maximum at a point $(t^\eps,x^\eps)$ satisfying $(t^\eps > 0,x^\eps)\to(t^0,x^0)$ as $\eps\to0$. The conclusion is the same.

\bigskip

\bigskip

We shall now prove that the initial condition for $\min_{w}\varphi_*$ is given by
\begin{equation}\label{eq:inicondsuper}
\max\left( \min\left \lbrace \partial_t \left( \min_{w}\varphi_*\right) + (1+r) \mathcal{H} \left(\frac{\nabla_x \left( \min_{w}\varphi_*\right)}{1+r} \right) + r , \min_{w}\varphi_* \right\rbrace,   \min_{w}\varphi_* - \min_{v \in V} \varphi_0 \right) \geq 0,
\end{equation}
on $\left\{ t = 0 \right\} \times \R^n$ in the viscosity sense.

Let us prove \eqref{eq:inicondsuper}. Let $\psi \in \mathcal{C}^1 \left( \R^+ \times \R \right)$ be a test function such that $\min_{w \in V}\varphi_* - \psi$ has a strict local minimum at $(t^0 = 0,x^0)$. We now prove that either
\begin{equation*} 
\min_{w \in V}\varphi_*(0,x^0,w) \geq \min_{v \in V} \varphi_0(x^0,v),
\end{equation*}
or 
\begin{equation*}
\partial_t \psi + (1+r) \mathcal{H} \left(\frac{\nabla_x \psi}{1+r} \right) + r \geq 0 \qquad \textrm{ and } \qquad \min_{w \in V}\varphi_*(0,x^0,w) \geq 0.
\end{equation*}

Suppose that $\min_{w \in V}\varphi_*(0,x^0,w) < \min_{v \in V} \varphi_0(x^0,v)$. The argument now starts similarly as in the proof above. Let us define the perturbed test function $\psi^{\eps}:=\psi +\eps\eta_\delta$. For any $\delta>0$, the function $\psi^\eps$ converges uniformly to $\psi$ since $\eta_\delta$ is bounded on $V$. Since $\varphi_*(0,x^0,\cdot)$ attains its minimum at, say, $v^0$, we have, for all $v \in V$ and locally in the $(t,x)$ variables, 
\begin{equation*}
\varphi_*(0,x^0,v_0) - \psi(0,x^0) \leq \min_{w \in V} \varphi_*(0,x^0) - \psi(0,x^0) \leq \min_{w \in V} \varphi_*(t,x) - \psi(t,x) \leq \varphi_*(t,x,v) - \psi(t,x),
\end{equation*}
and thus $(0,x^0,v^0)$ is a local minimum of $\varphi_* - \psi$, strict in the $(t,x)$ variables. By the definition of the lower semi-limit, there exists a sequence $\left( t_\delta^{\eps},x_\delta^{\eps},v_\delta^{\eps}\right)$ of minimum points of $\varphi^{\eps}-\psi^{\eps}$ satisfying $(t_\delta^{\eps},x_\delta^{\eps})\to(0,x^0)$. We first claim that there exists a subsequence $(t_{\eps_{k}},x_{\eps_{k}},v_{\eps_{k}})_k$ of the above sequence, with $\eps_k\to 0$ as $k\to \infty$, such that $t_{\eps_{k}} > 0$, for all $k$.

Suppose that this is not true. Then, take a sequence $(x_\delta^{\eps_{k'}},v_\delta^{\eps_{k'}})_{k'}$ such that $(\eps_{k'},x_\delta^{\eps_{k'}})\to (0,x^0)$ and 
that $\varphi^{\eps_{k'}}-\psi^{\eps_{k'}}$ has a local minimum at $\left(0,x_\delta^{\eps_{k'}},v_\delta^{\eps_{k'}}\right)$. It follows that, for all $(t,x,v)$ in some neighborhood of $(0,x_\delta^{\eps_{k'}},v_\delta^{\eps_{k'}})$, we have
\begin{align*}
\min_{v \in V} \varphi_0(x_\delta^{\eps_{k'}},v) - \psi^{\eps_{k'}}\left( 0,x_\delta^{\eps_{k'}},v_\delta^{\eps_{k'}}\right) & \leq \varphi_0(x_\delta^{\eps_{k'}},v_\delta^{\eps_{k'}}) - \psi^{\eps_{k'}}\left( 0,x_\delta^{\eps_{k'}},v_\delta^{\eps_{k'}}\right) \\ &\leq \varphi^{\eps_{k'}} \left( 0,x_\delta^{\eps_{k'}},v_\delta^{\eps_{k'}}\right) - \psi^{\eps_{k'}}\left( 0,x_\delta^{\eps_{k'}},v_\delta^{\eps_{k'}}\right)\\ &\leq \varphi^{\eps_{k'}} \left( t , x , v \right) - \psi^{\eps_{k'}}\left( t , x ,v \right).
\end{align*}
Taking $ \underset{\underset{(t,x,v)\to (0,x^0,v_0)}{k'\to \infty}}{\liminf}$ at the both sides of the inequality, one obtains \begin{equation*}
\min_{v \in V} \varphi_0(x^0,v) - \psi \left( 0 , x^0 \right) \leq \min_{w \in V}\varphi_*(0,x^0) - \psi \left( 0 , x^0 \right).
\end{equation*}
However, this is in  contradiction with $\min_{w \in V}\varphi_*(0,x^0,w) < \min_{v \in V} \varphi_0(x^0,v)$. Now having in hand that this sequence of times $t_{\eps_n}>0$, one can reproduce the same argument as from the proof above along the subsequence $(t_{\eps_n},x_{\eps_n},v_{\eps_n})$.

By the strong uniqueness principle satisfied by \eqref{eq:varHJ} (that is, a comparison principle for discontinuous sub- and super- solutions), we deduce that for all $(t,x,v) \in \R_+^* \times \R^n \times V$,
\begin{equation*}
\min_{w \in V} \varphi_*(t,x,w) \leq \varphi_*(t,x,v) \leq \varphi^*(t,x,v) = \varphi^*(t,x) \leq \min_{w \in V} \varphi_*(t,x,w)
\end{equation*}
We deduce that necessarily all these inequalities are equalities, and thus that $\varphi^\eps$ converges locally uniformly towards $\varphi^0$, independent of $v$, on any subcompact of $\R_+^* \times \R^n$.
\end{proof}

\subsection{Convergence of the macroscopic density $\rho^\eps$}

We prove a convergence result for $\rho^\eps$ in the region $\lbrace \varphi^0 = 0 \rbrace$. Namely

\begin{proposition}\label{prop:zones}
Let $\varphi^\eps$ be the solution of \eqref{eq:mainHJeps}. Then, uniformly on compact subsets of $\text{Int} \left\lbrace \varphi^0 = 0\right\rbrace $, 
\begin{equation*}
\lim_{\eps \to 0} \rho^\eps = 1, \qquad \lim_{\eps \to 0} f^\eps \left( \cdot , v \right) = M(v).
\end{equation*}
\end{proposition}

\begin{proof}[\bf Proof of Proposition \ref{prop:zones}]

We develop similar arguments as in \cite{evans_pde_1989}. Let $K$ be a compact set of $\lbrace \varphi^0=0\rbrace$. Note that it suffices to prove the result when $K$ is a cylinder. Let $(t^0,x^0) \in \text{Int}\left( K \right)$ and the test function 
\begin{equation*}
\forall (t,x) \in K, \qquad \psi^0(t,x) = \vert x - x^0 \vert^2 + \left( t - t^0 \right)^2.
\end{equation*}
Since $\varphi^0 = 0$ on $K$, the function $\varphi^0 -  \psi^0$ admits a strict maximum in $(t^0,x^0)$. The locally uniform convergence of $\varphi^\eps - \psi^0$ gives a sequence $(t^{\eps},x^{\eps},v^{\eps})$ of maximum points with $(t^{\eps},x^{\eps}) \to (t^0,x^0)$ and a bounded sequence $v^{\eps}$ such that at the point $(t^{\eps},x^{\eps},v^{\eps})$ one has:
\begin{equation}\label{nimp}
\partial_t \psi^0 + v^{\eps} \cdot \nabla_x \psi^0 + r \leq r \rho^{\eps}.
\end{equation}
As a consequence, one has, since $r>0$, 
\begin{equation}\label{convepsrho}
\rho^\eps(t^{\eps},x^{\eps}) \geq 1 + o(1), \qquad \textrm{as } {\eps \to 0},
\end{equation}
and then $\lim_{\eps \to 0}  \rho^\eps(t^{\eps},x^{\eps}) = 1$ if one recalls $\rho^\eps \leq 1$ (which, again, is a consequence of the maximum principle). 

However, we need an extra argument to get $\lim_{\eps \to 0}  \rho^\eps(t^0,x^0) = 1$. Since $(t^{\eps},x^{\eps},v^{\eps})$ maximizes $\varphi^\eps - \psi^0$, we deduce that for all $v\in V$, we have
\begin{equation*}
\varphi^\eps\left(t^{\eps},x^{\eps},v^{\eps}\right) - \psi^0(t^\eps,x^\eps) \geq \varphi^\eps\left(t^0,x^0,v\right) - \psi^0(t^0,x^0). 
\end{equation*}
Since $\psi^0(t^\eps,x^\eps) \geq 0$, $ \psi^0(t^0,x^0) = 0$, we find
\begin{equation}\label{eq:subf}
f^\eps(t^0,x^0,v) = M(v) e^{- \frac{\varphi^\eps(t^0,x^0,v)}{\eps}} \geq M(v) e^{- \frac{\varphi^\eps(t^{\eps},x^{\eps},v^{\eps})}{\eps}}. 
\end{equation} 
We shall now prove that $\lim_{\eps \to 0} \eps^{-1} \varphi^\eps(t^{\eps},x^{\eps},v^{\eps}) = 0$. Let us rewrite \eqref{eq:mainHJeps} at the point $(t^{\eps},x^{\eps},v^{\eps})$ in the form
\begin{equation*}
r \rho^\eps(t^\eps,x^\eps)\left( e^{\frac{\varphi^\eps(t^{\eps},x^{\eps},v^{\eps})}{\eps}} - 1\right) = \left(1-\int_V M(v') e^{\frac{\varphi^{\eps}(t^{\eps},x^{\eps},v^{\eps}) -\varphi^{\eps}(t^{\eps},x^{\eps},v')}{\eps}} dv'\right) - \left( \partial_t \psi^0 + v \cdot \nabla_x \psi^0 \right)(t^{\eps},x^{\eps},v^{\eps})
\end{equation*}
We finally deduce using the maximum principle in the latter r.h.s. that
\begin{equation*}
0 \leq r \rho^\eps(t^\eps,x^\eps)\left( e^{\frac{\varphi^\eps(t^{\eps},x^{\eps},v^{\eps})}{\eps}} - 1\right) \leq  o_{\eps \to 0}(1)
\end{equation*}
and thus $\lim_{\eps \to 0} \left( \eps^{-1} \varphi^\eps(t^{\eps},x^{\eps},v^{\eps}) \right) = 0$. This implies $\lim_{\eps \to 0} f^\eps(t,x,v) = M(v)$ locally uniformly on $K \times V$.

\end{proof}

\subsection{Speed of expansion}

To be self-contained, we recall here how to study the propagation of the front after deriving the limit variational equation, in the case $r>0$. From Evans and Souganidis \cite{evans_pde_1989}, we are able to identify the solution of the variational Hamilton-Jacobi equation \eqref{eq:varHJ} using the Lagrangian duality. We emphasize that, in this context, one may assume that our initial condition is well-prepared, \em i.e. \em $\varphi(0,x,v)=\varphi_0(x)$. We recall the equation:
\begin{equation*}
\begin{cases}
\min\left \lbrace \partial_t \varphi + (1+r)\mathcal{H} \left(\frac{\nabla_x \varphi}{1+r} \right) + r, \varphi \right\rbrace = 0, \qquad  \forall (t,x) \in \R_+^* \times \R^n, \medskip\\
\varphi(0,x) = \varphi_0(x).
\end{cases}
\end{equation*}
We recall from \cite{bouin_kinetic_2012,caillerie_large_2017} that the Hamiltonian $\mathcal{H}$ is convex. For any $e_0 \in \mathbb{S}^{n-1}$, we define the minimal speed in that direction by the formula
\begin{equation}\label{eq:linspeed}
c^*(e) = \inf_{\lambda > 0} c(\lambda,e), \quad c(\lambda,e) = \frac{1}{\lambda}  \left( (1+r)\mathcal{H} \left(\frac{\lambda e}{1+r} \right) + r\right).
\end{equation}


We first discuss the speed of propagation of a front-like initial data. 

\begin{proposition}\label{prop:nullsetfront}

Assume that 
\begin{equation*}
\varphi_0(x) := \left\lbrace\begin{array}{lcl}
0& x \in e_0^\bot\\
+ \infty & \text{ else}\\
\end{array}\right.. 
\end{equation*}
Then the nullset of $\varphi$ propagates at speed $c^*(e_0)$ :
\begin{equation*}
\forall t \geq 0, \qquad \left\{ \varphi(t, \cdot ) = 0 \right\} = e_0^\bot + c^*(e_0) t \, e_0.
\end{equation*}
\end{proposition}

\begin{proof}[\bf Proof of Proposition \ref{prop:nullsetfront}]

We first notice that since the initial data is invariant under any translation in $e_0^\bot$, and the  the equation \eqref{eq:varHJ} invariant by translation, the solution $\varphi$ depends only on $x\cdot e_0$. That is $\varphi(t,x) = \varphi(t,(x \cdot e_0)e_0) = \overline{\varphi}(t,x \cdot e_0)$.  
The function $\overline{\varphi}$ satisfies
\begin{equation*}
\begin{cases}
\min\left \lbrace \partial_t \overline{\varphi} + (1+r)\mathcal{H} \left(\frac{\partial_\xi \overline{\varphi}}{1+r} e_0 \right) + r, \varphi \right\rbrace = 0, \qquad  \forall (t,\xi) \in \R_+^* \times \R, \medskip\\
\overline{\varphi}(0,\xi) = \overline{\varphi}_0(\xi).
\end{cases}
\end{equation*}
where 
\begin{equation*}
\overline{\varphi}_0(\xi) := \left\lbrace\begin{array}{lcl}
0& \xi = 0,\\
+ \infty & \text{ else}\\
\end{array}\right.. 
\end{equation*}
The Lagrangian associated to the latter Hamilton-Jacobi equation is by definition 
\begin{align*}
 \mathcal{L}(p) &= \sup_{q \in \R} \left( p q - (1+r)\mathcal{H}\left(\frac{q}{1+r} e_0\right) - r  \right),\\
&= \sup_{q \in \R} \left( p q - (1+r)\mathcal{H}\left(\frac{\vert q \vert}{1+r} e_0\right) - r  \right), \\
&= \sup_{q \in \R} \left( p q - \vert q \vert c(\vert q \vert,e_0) \right),
\end{align*}
To solve the variational Hamilton-Jacobi equation, let us define
\begin{equation*}
J(x,t) = \inf_{x \in X} \left \lbrace \int_{0}^{t}  \left[ \mathcal{L}( \dot x ) \right] ds \,  \big\vert \, x(0) = x, x(t) = 0 \right \rbrace 
\end{equation*}
the minimizer of the action associated to the Lagrangian. Thanks to the so-called Freidlin condition, see \cite{evans_pde_1989,freidlin_functional_1985} we deduce that the solution of \eqref{eq:varHJ} is 
\begin{equation*}
\overline{\varphi}(t,\xi) = \max \left( J(\xi,t) , 0 \right) . 
\end{equation*}
The Hopf-Lax formula gives $J(\xi,t) = t \mathcal{L}\left( t^{-1}\xi\right)$
thanks to the assumption on the initial condition. Hence,
\begin{align*}
\xi \in \left\{ \overline{\varphi}(t, \cdot ) = 0 \right\} \Longleftrightarrow \mathcal{L}\left(t^{-1}\xi\right) \leq 0 \quad \Longleftrightarrow &\quad \sup_{q \in \R} \left( q \xi - \vert q \vert c(\vert q \vert, e_0) t \right) \leq 0, \\
\quad \Longleftrightarrow &\quad \forall q  \in \R, \quad q \xi - \vert q \vert c(\vert q \vert,e_0) t \leq 0,\\
\quad \Longleftrightarrow &\quad  \vert \xi \vert \leq  c^*(e_0) t.
\end{align*}
We deduce the result for $\varphi$ by changing the variables back.
\end{proof}

For a compactly supported initial data, the issue of the speed of propagation in general is more involved, since different directions may have different speeds of propagation. Namely, the following Freidlin-G\"artner formula holds:

\begin{proposition}\label{prop:nullsetfreidlin_functional_1985}
Assume that 
\begin{equation*}
\varphi_0(x) := \left\lbrace\begin{array}{lcl}
0& x=0\\
+ \infty & \text{ else}\\
\end{array}\right.. 
\end{equation*}
Define 
\begin{equation*}
w^*(e_0) = \min_{\substack{e \in \mathbb{S}^{n-1}\\e_0 \cdot e > 0}}\left( \frac{c^*(e)}{e_0 \cdot e}\right).
\end{equation*}
Then the nullset of $\varphi$ propagates at speed $w^*(e_0)$ in the direction $e_0$ :
\begin{equation*}
\forall t \geq 0, \qquad \left\{ x \in \R, \; \varphi(t,x \, e_0) = 0 \right\} = \left\{ x \in \R, \; \vert x \vert \leq w^*(e_0) t  \right\}.
\end{equation*}

\end{proposition}

\begin{proof}[\bf Proof of Proposition \ref{prop:nullsetfreidlin_functional_1985}]

The Lagrangian is by definition 
\begin{align*}
 \mathcal{L}(p) &= \sup_{q \in \R^n} \left( p \cdot q - (1+r)\mathcal{H}\left(\frac{q}{1+r}\right) - r  \right), \\
&= \sup_{e \in \mathbb{S}^{n-1}} \sup_{\lambda \in \R^+} \left( \lambda p \cdot e - \left[ (1+r)\mathcal{H}\left(\frac{\lambda e}{1+r}\right) + r \right] \right),\\
&= \sup_{e \in \mathbb{S}^{n-1}} \sup_{\lambda \in \R^+} \left( \lambda \left[ p \cdot e - c(\lambda,e) \right]\right),
\end{align*}
To solve the variational Hamilton-Jacobi equation, let us define
\begin{equation*}
J(x,t) = \inf_{x \in X} \left \lbrace \int_{0}^{t}  \left[ \mathcal{L}( \dot x ) \right] ds \,  \big\vert \, x(0) = x, x(t) = 0 \right \rbrace 
\end{equation*}
the minimizer of the action associated to the Lagrangian. Thanks to the so-called Freidlin condition, see \cite{evans_pde_1989,freidlin_functional_1985} we deduce that the solution of \eqref{eq:varHJ} is 
\begin{equation*}
\varphi(x,t) = \max \left( J(x,t) , 0 \right) . 
\end{equation*}
The Lax formula gives 
\begin{equation*}
J(x,t) = \min_{y \in \R^n} \left\{ t \mathcal{L}\left( \frac{x-y}{t} \right) + \varphi_0(y) \right\} = t \mathcal{L}\left( \frac{x}{t} \right)
\end{equation*}
thanks to the assumption on the initial condition. Hence,
\begin{align*}
\varphi(t, x e_0 ) = 0  \Longleftrightarrow \mathcal{L}\left(\frac{x}{t} e_0 \right) \leq 0 \quad  \Longleftrightarrow &\quad \sup_{e \in \mathbb{S}^{n-1}} \sup_{\lambda \in \R^+} \left( \lambda \left[ x (e_0 \cdot e) - c_e(\lambda) t \right]\right) \leq 0, \\
\quad \Longleftrightarrow &\quad \forall \lambda  \in \R^+, \forall e \in \mathbb{S}^{n-1}, \quad  \lambda \left[ (x\cdot e_0)(e_0 \cdot e) - c_e(\lambda) t \right] \leq 0,\\
\quad \Longleftrightarrow &\quad  \forall e \in \mathbb{S}^{n-1}, \quad  x(e_0 \cdot e) \leq  c^*(e) t\\
\quad \Longleftrightarrow &\quad  \vert x\vert \leq  \min_{\substack{e \in \mathbb{S}^{n-1}\\e_0 \cdot e > 0}}\left( \frac{c^*(e)}{e_0 \cdot e}\right) t = w^*(e_0) t.
\end{align*}
\end{proof}

\section{Existence of travelling waves and spreading result}\label{sec:TW}

In this Section, we now explain how to construct travelling wave solutions to \eqref{eq:main}. We will follow closely the construction in \cite{bouin_propagation_2015}. As is classical in this type of Fisher-KPP problems, the speeds of propagation are given by studying the linearized problem at infinity. As we will see later on, the main difference that has motivated this paper is the possible singularity of $c(\lambda,e)$ at $\lambda^*(e)$. 


\subsection{Proof of Theorem \ref{thm:existence-tw} : Travelling wave solutions}

Given a direction $e \in \mathbb{S}^{n-1}$, looking for  exponential solutions to the linearized problem of the form $e^{-\lambda\left(x\cdot e - c(\lambda,e) t\right)} F_{\lambda,e}(v)$ for any positive $\lambda$ is exactly looking for solutions to 
\begin{equation*}
\left[ 1 + \lambda ( c(\lambda,e) - v \cdot e) \right] F_{\lambda,e}(v) = (1+r) M(v) \int_V F_{\lambda,e}(v') dv', \qquad v \in V.
\end{equation*}
In view of earlier computations, it boils down to setting $c(\lambda,e)$ as in \eqref{eq:linspeed} and $F_{\lambda,e} := \tilde Q_{\frac{\lambda e}{1+r}}$ as in \eqref{eq:profile}. 

Recall that $\frac{\lambda e}{1+r} \in \text{Sing}(M)$ if and only if $l(e)\leq \frac{\lambda}{1+r}$, that is $\lambda \geq \tilde\lambda(e) := (1+r) l(e)$. Thus, for $\lambda \leq \tilde\lambda(e)$, the function $c(\lambda,e)$ is convex and regular, and the profile is explicitly given by
\begin{equation*}
F_{\lambda,e}\left(v\right)=\frac{(1+r) M\left(v\right)}{1+ \lambda ( c(\lambda,e) - v\cdot e)} > 0.
\end{equation*}
For $\lambda \geq \tilde \lambda(e)$, that is to say $\frac{\lambda e}{1+r} \in \text{Sing}(M)$ one has $c(\lambda,e) = \overline{v}(e) - \frac{1}{\lambda}$ which is concave and increasing. As such, the infimum of $\lambda \mapsto c(\lambda,e)$ is attained for a $\lambda \leq \tilde\lambda(e)$, which we denote $\lambda^*(e)$. As a consequence, the minimal speed $c^*(e)$ is always associated to an integrable eigenvector, since if $\lambda^*(e) = \tilde \lambda(e)$, one has
\begin{equation*}
F_{\tilde\lambda(e),e}\left(v\right) = \frac{(1+r) M\left(v\right)}{\tilde\lambda(e) \left( \bar v(e) - v\cdot e \right)},
\end{equation*}
with $\int_V F_{\tilde\lambda(e),e}\left(v\right) dv = 1$ thanks to the definition of $\tilde\lambda(e)$. 

Given a direction $e \in \mathbb{S}^{n-1}$, we shall now discuss the type of functions $\lambda \mapsto c(\lambda,e)$ that may arise from this problem. Qualitatively, four situations may happen. The first possibility is the one already appearing in \cite{bouin_propagation_2015} in the mono-dimensional case, that is $\tilde{\lambda}(e)=+\infty$ and thus $\mathrm{Sing}(M)=\emptyset$. We plot an exemple of this case in Figure \ref{fig:Shape}, case 1. If $\tilde{\lambda}(e) < +\infty$, three supplementary situations can occur. Either the infimum of $\lambda \mapsto c(\lambda,e)$ is attained for $\lambda<\tilde{\lambda}(e)$, as shown in Figure \ref{fig:Shape}, case 2, either it is attained for $\lambda=\tilde{\lambda}(e)$. In the latter case, the infimum can either be attained at a point where the left derivative of $c(\lambda,e)$ is zero (Figure \ref{fig:Shape}, case 3), or where it is negative (Figure \ref{fig:Shape}, case 4).

\begin{figure}[htbp]
\label{fig:Shape}
\begin{center}
\subfigure[Case $1$]
{
\includegraphics[width=0.40\linewidth]{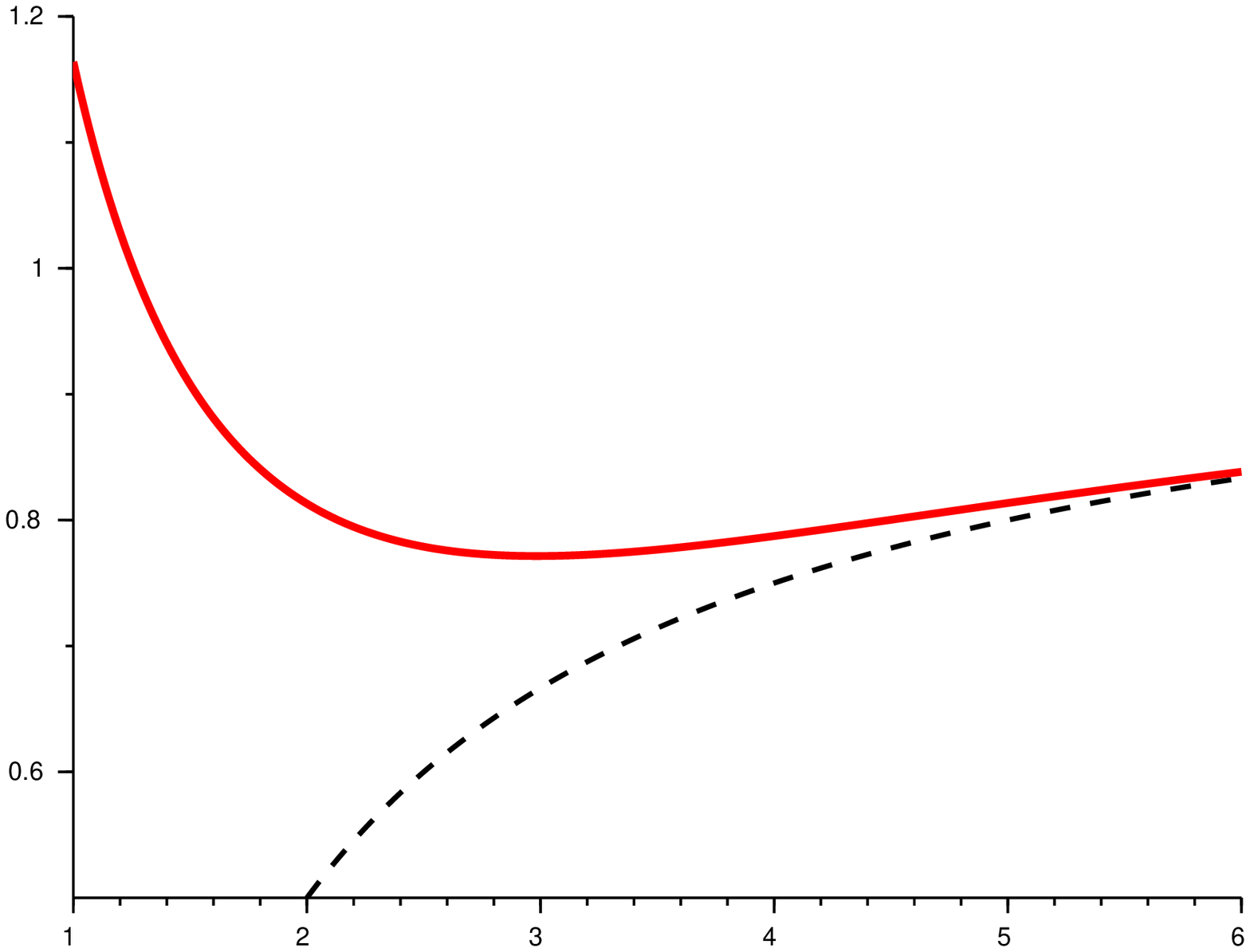} 
} 
\subfigure[Case $2$]
{
\includegraphics[width=0.40\linewidth]{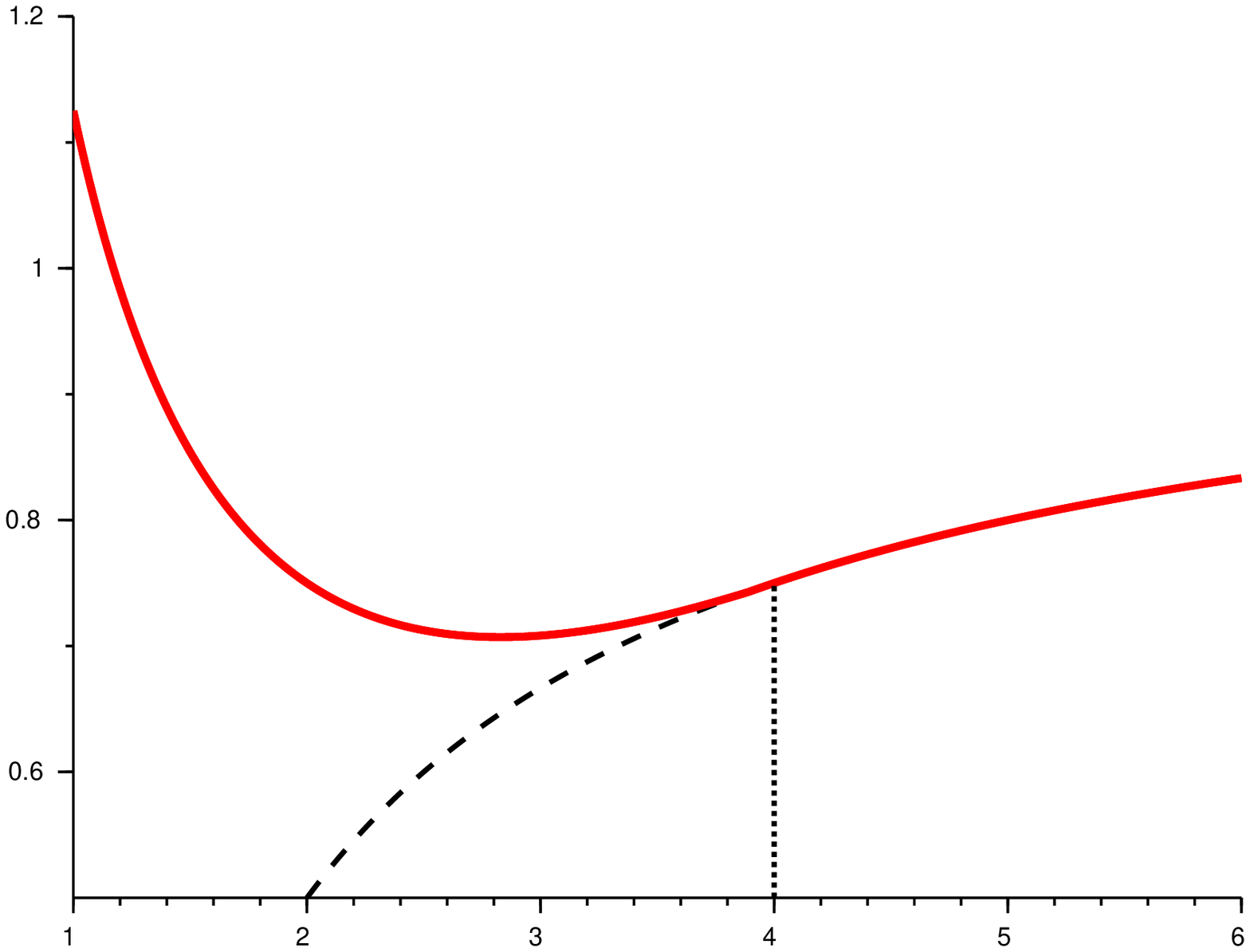}
} 
\subfigure[Case $3$]
{
\includegraphics[width=0.40\linewidth]{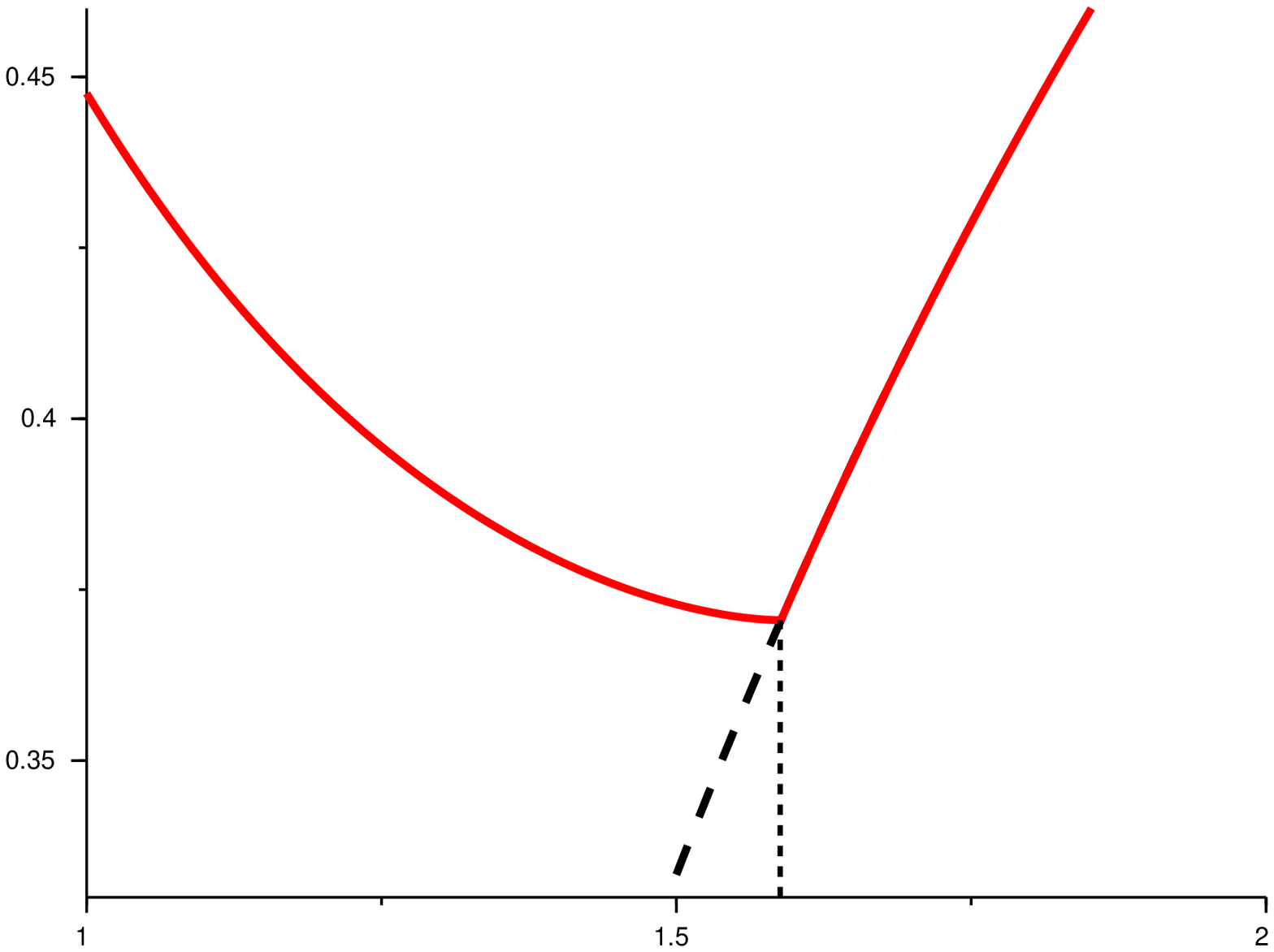}
} 
\subfigure[Case $4$]
{
\includegraphics[width=0.40\linewidth]{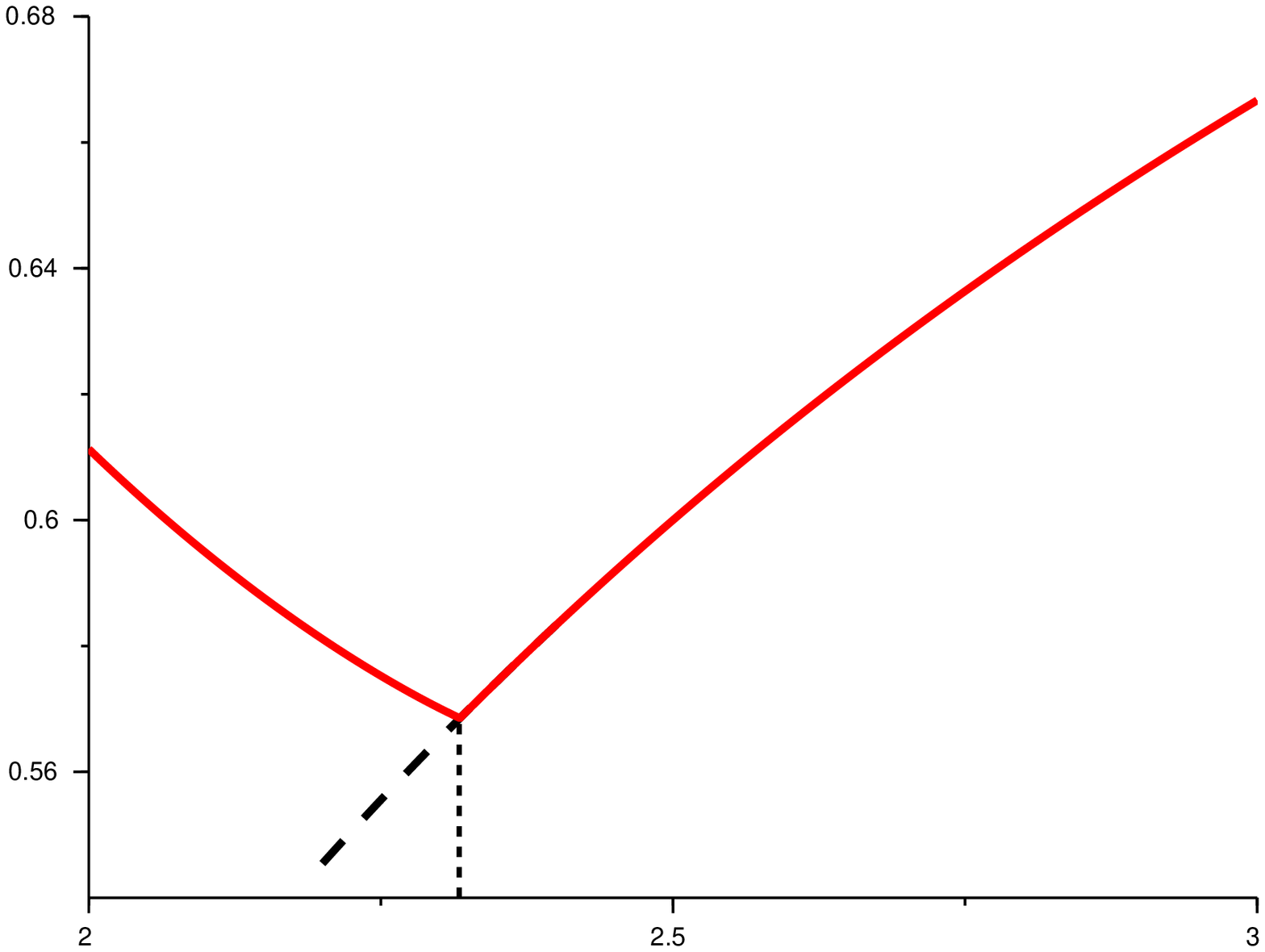}
} 
\end{center}
\caption{Various cases of speed functions $\lambda \mapsto c(\lambda,e)$. Red plain line: $\lambda \mapsto c(\lambda,e)$. Black dotted line: $\lambda \mapsto \bar v(e) - \frac{1}{\lambda}$. (a) $n=1$, $V=[-1,1]$, $e=1$, $M\equiv \frac{1}{2}$ and $r=1$. In this case, $\mathrm{Sing}(M)=\emptyset$ so that the function $\lambda\mapsto c(\lambda,1)$ is regular. This is the case discussed in \cite{bouin_propagation_2015}. (b) $n=2$, $V=D(0,1)$, $M\equiv \frac{1}{\pi}$ and $r=1$. In this case, $\mathrm{Sing}(M)\neq \emptyset$ but the minimum of $c(\lambda,e_1)$ is attained for $\lambda<\tilde{\lambda}(e_1)=4$. (c) $n=1$, $V=[-1,1]$, $M(v)=\frac{3}{2}(1-|v|)^2$ and $r=-1+l(1)^{-2}\int_{-1}^1 \frac{M(v)}{(1-v)^2}dv\approx 0.37$. In this case, the minimum of $c(\lambda,e_1)$ is attained for $\lambda=\tilde{\lambda}(e_1)$, with a zero left derivative. Numerically $\tilde{\lambda}(1)\approx 1.58$. (d) $n=1$, $V=[-1,1]$, $M(v)=\frac{3}{2}(1-|v|)^2$ and $r=1$. In this case, the minimum of $c(\lambda,1)$ is attained for $\lambda=\tilde{\lambda}(1)$, with a negative left derivative. Numerically, $\tilde{\lambda}(1)\approx 2.31$.}
\end{figure}

\begin{rem}
One can get a criterion to check which case holds. The dispersion relation defining $c(\lambda,e)$ on $(0,\tilde\lambda(e))$ is 
\begin{equation*}
\mathcal{I}(\lambda,c(\lambda,e),e) = 1,
\end{equation*}
where
\begin{equation}\label{dispersionrelation}
\mathcal{I}(\lambda,c,e):=\int_V \frac{(1+r) M\left(v\right)}{1+ \lambda ( c - v\cdot e)} dv.
\end{equation}
Differentiating with respect to $\lambda$, we find 
\begin{equation*}
\int_V \frac{\lambda c'(\lambda,e) M\left(v\right)}{\left[1+ \lambda ( c(\lambda,e) - v\cdot e)\right]^2} dv + \int_V \frac{( c(\lambda,e) - v\cdot e)M\left(v\right)}{\left[1+ \lambda ( c(\lambda,e) - v\cdot e)\right]^2} dv = 0
\end{equation*}
Recalling $\int_V \frac{M\left(v\right)}{1+ \lambda ( c(\lambda,e) - v\cdot e)} dv = (1+r)^{-1}$ and defining
\begin{equation*}
\mathcal{J}(\lambda,e) = \int_V \frac{M\left(v\right)}{\left[1+ \lambda ( c(\lambda,e) - v\cdot e)\right]^2} dv,
\end{equation*}
we get 
\begin{equation*}
c'(\lambda,e) = \left( 1 - \frac{(1+r)^{-1}}{J(\lambda,e)} \right) \frac{1}{\lambda^2}.
\end{equation*}
As such, computing the value of $\lim_{\lambda \to \tilde\lambda(e)^-} \mathcal{J}(\lambda,e)$ allows to know in which case one falls. Indeed, the function $\lambda \mapsto c(\lambda)$ attains its minimum at $\tilde\lambda(e)$ if and only if $c'\left(\tilde\lambda^-(e)\right) \leq 0$, which is equivalent to $\mathcal{J}(\tilde\lambda(e)) \leq (1+r)^{-1}$ which is in turn equivalent to
\begin{equation}\label{eq:square}
\int_V \frac{M\left(v\right)}{\left( \bar v(e) - v\cdot e \right)^2} dv \leq (1+r) l(e)^2 ,
\end{equation} 
which can be checked case by case. Note that one has always, given the Cauchy-Schwarz inequality, $l(e)^2 \leq \int_V \frac{M\left(v\right)}{\left( \bar v(e) - v\cdot e \right)^2} dv$.
\end{rem}

\begin{example}
Let us look back at \Cref{ex:annulebord}. As was stated, $l(1)=3(2\ln(2)-1)$ and $\int_V \frac{M(v)}{(1 - v)^2}dv=6(1-\ln(2))<+\infty$. Thus, for $r >-1+ l(1)^{-2}\int \frac{M(v)}{(1 - v)^2}dv>0$, the condition (\ref{eq:square}) is satisfied so the minimum of $\lambda \mapsto c(\lambda,e)$ is attained at $\tilde{\lambda}(e)$. For $r=-1 + l(1)^{-2}\int_V \frac{M(v)}{(1 - v)^2}dv$ the minimum has its left derivative equal to 0 (\em i.e. \em $\lambda^*(1)=\tilde{\lambda}(1)$). We illustrate those results in \cref{fig:Shape}, case 3 and 4.
\end{example}

Since $c(\lambda,e)$ tends to infinity when $\lambda$ tends to $0$, for any $c \geq c^*(e)$ one can find $\lambda \in (0, \tilde\lambda(e)]$ such that $c(\lambda,e) = c$.

Fix $c\in \left( c^*(e), \overline{v}(e)  \right)$. Denote $\lambda_c$ is the smallest solution in $(0,\tilde\lambda(e))$ of $c(\lambda_c,e) = c$. Notice that by construction it is possible to obtain $F_{\lambda_c,e}$ integrable and bounded (bounded since $c > c^*(e))$, the proof of \cite{bouin_propagation_2015}, Section 3.2, that constructs sub and super solutions for \eqref{eq:main} is unchanged. From the construction of a pair of sub- and super-solutions, we deduce the existence of travelling wave solutions exactly as in \cite{bouin_propagation_2015}, by a monotonicity method when $c > c^*(e)$ and passing to the limit $c \to c^*(e)$ to get the case $c=c^*(e)$.

The main difference between the mono-dimensional case of \cite{bouin_propagation_2015} and the higher dimensional case comes here. It is rather non-standard and interesting that the function giving the speed of propagation could be singular at its minimum value. 

To prove that $c^*$ is still the minimal speed of propagation, the arguments used in \cite[Lemma 3.10]{bouin_propagation_2015} are not applicable. These arguments can be summarized as follows : in the one dimensional case when $M\geq\delta>0$, the function $\lambda \mapsto \mathcal{I}(\lambda,c,e)$ (recall (\ref{dispersionrelation})) is analytic. Thus, we can not find $\lambda>0$ such that $\mathcal{I}(\lambda,c,e)=1$ when $c<c^*$. However, an argument using the Rouch\'e Theorem states that we can solve this problem in $\mathbb{C}\setminus \R$. Assuming that there exists a travelling wave solution $f$ for $c<c^*$, we then can use such a $\lambda\in \mathbb{C}$ to construct a subsolution under $f$ which dos not converge to 0 as $x\to \infty$. In our framework, the function $\lambda\mapsto\mathcal{I}(\lambda,c,e)$ might not be analytic around $\lambda^*(e)$, which prevents us from using this technique. We thus choose to use the Hamilton-Jacobi framework combined to the comparaison principle.

We now prove the following lemma.

\begin{lem}\label{lem:minspeed}
Let $f$ be a travelling wave solution to \eqref{eq:main} in the direction $e \in \mathbb{S}^{n-1}$, with speed $c$. Then $c\geq c^*(e)$.
\end{lem}
\begin{proof}[{\bf Proof of \Cref{lem:minspeed}}]
\NC{Let $f$ be such a travelling wave solution with initial data $\tilde f(x,v)$, \em i.e \em $f(t,x,v)=\tilde{f}(x\cdot e - c t,v)$. After \Cref{prop:zones}, we deduce that $f^\eps(t,x,v) = \tilde{f}\left(\frac{1}{\eps}\left( x\cdot e - c t \right),v\right)$} satisfies $\lim_{\eps \to 0} f^\eps = M$ on $x\cdot e - c t < 0$ and $\lim_{\eps \to 0} f^\eps = 0$ on $x\cdot e - c t > 0$.
Take $0<\gamma < 1$ and define $g(x,v) = \gamma M(v) \textbf{1}_{[-1,1]\times \R^{n-1}}(x)$ and $g^{\eps}(x,v)=g(x/\eps,v)$. We have
\begin{equation*}
\psi^\eps(x) = - \eps \ln(g(x/\eps,v)/M) = \left\lbrace\begin{array}{lcl}
- \eps \ln(\gamma) & x \in [-\eps,\eps] e_0 + e_0^\bot\\
+ \infty & \text{ else}\\
\end{array}\right.. 
\end{equation*}
Since $\lim_{z\to -\infty} \tilde f(z,v)=M$ uniformly in $v \in V$, one can shift the profile sufficiently enough so that $M \geq f \geq g \geq 0$. Thus, the comparison principle (see \cite{bouin_propagation_2015}, Proposition 2.2 for a proof) yields that $f^\eps \geq g^\eps$. Passing to the limit $\eps \to 0$, and recalling \Cref{thm:HJlimit}, \Cref{prop:zones} and \Cref{prop:nullsetfront}, we deduce that 
\begin{equation*}
\left( e_0^\bot + c^*(e_0) t e_0 \right) \cdot e_0 - ct \leq 0,
\end{equation*}
from which the result follows.
\end{proof}

From the Hamilton-Jacobi formalism, we may also deduce the following. 

\begin{proof}[{\bf Proof of Proposition \ref{prop:spreadingplanar}}]
We start by proving \eqref{eq:frontplanar}. For this, we use the the super-solution naturally provided by the linearized problem. We have 
\begin{equation*}
f(t,x,v) \leq \min\lbrace M(v) , e^{-\lambda^*(e_0)\left(x\cdot e_0 - c^*(e_0) t\right)} F_{\lambda^*(e_0),e_0}(v) \rbrace
\end{equation*}
As a consequence,
\begin{equation*}
\rho(t,x) \leq \min\lbrace 1 , e^{-\lambda^*(e_0)\left(x\cdot e_0 - c^*(e_0) t\right)} \rbrace,
\end{equation*}
and thus one has $\lim_{t \to +\infty} \sup_{x\cdot e_0 > ct} \rho(t,x) = 0$.

For \eqref{eq:backplanar}, we use the Hamilton-Jacobi results in the following way. We first notice that since the initial data is invariant under any translation in $e_0^\bot$, and the  the equation \eqref{eq:varHJ} invariant by translation, the solution $f(t,x,v)$ depends only on $x\cdot e_0$. That is $f(t,x,v) = f(t,(x \cdot e_0)e_0,v) = \tilde{f}(t,x \cdot e_0,v)$. For any $c < c^*(e_0)$, recalling \Cref{thm:HJlimit}, \Cref{prop:zones} and \Cref{prop:nullsetfront}, we have
\begin{equation*}
\lim_{t \to \infty} f(t,e_0^\bot + c t e_0,v) = \lim_{t \to \infty} \tilde f(t,ct,v) = \lim_{\eps \to 0} \tilde f^\eps(1,c,v) =  M(v),
\end{equation*} 
since $c < c^*(e_0)$.
\end{proof}

\subsection{Proof of \Cref{prop:spreadingbounded} : spreading of a compactly supported initial data}

We finally prove \Cref{prop:spreadingbounded}. The spreading result \eqref{eq:frontcompact} goes as for the Fisher-KPP equation in an heterogeneous media \cite{berestycki_spreading_2012}. It can be found by using the super solution 
\begin{equation*}
\overline{f}(t,x,v) = \inf_{e \in \mathbb{S}^{n-1}} e^{-\lambda^*(e)\left(x\cdot e - c^*(
e) t\right)} Q_{\lambda^*(e) e}(v)
\end{equation*}
By the comparison principle, and since the initial data is compactly supported, the function $\overline{f}$ lies above $f$ (multiplying $\overline{f}$ by a big constant if necessary). We deduce that for any given $e_0 \in \mathbb{S}^{n-1}$, and any fixed $x \in \R^n$,
\begin{equation*}
f(t,x+c e_0 t,v) \leq \inf_{e \in \mathbb{S}^{n-1}} e^{-\lambda^*(e)\left((x+c e_0 t)\cdot e - c^*(
e) t\right)} Q_{\lambda^*(e) e}(v) = \inf_{e \in \mathbb{S}^{n-1}} e^{-\lambda^*(e)\left(x\cdot e + c e_0\cdot e t - c^*(
e) t\right)} Q_{\lambda^*(e) e}(v).
\end{equation*}

Moreover, the domain of $Q_{\lambda^*(e)e}$ contains $V\setminus \left\{v_{max}e\right\}$ and $Q_{\lambda^*(e)e}$ is bounded on all compact sets of $V\setminus \left\{v_{max}e\right\}$. Hence, for fixed $v\in V$, we can choose $e\in \mathbb{S}^{n-1}$ such that $v \in V\setminus \left\{v_{max}e\right\}$. Then, as soon as $c > w^*(e_0)$, we have $c (e \cdot e_0) > c^*(e)$ for any $e$, and thus $\lim_{t \to \infty} f(t,x+c e_0 t,v) = 0$.

Moreover, we shall prove \eqref{eq:backcompact} as follows. For any $c < c^*(e_0)$, recalling \Cref{thm:HJlimit}, \Cref{prop:zones} and \Cref{prop:nullsetfreidlin_functional_1985}, we have
\begin{equation*}
\lim_{t \to \infty} f(t,c t e_0,v) = \lim_{t \to \infty} \tilde f(t,ct,v) = \lim_{\eps \to 0} \tilde f^\eps(1,ce_0,v) =  M(v),
\end{equation*} 
since $c < w^*(e_0)$.\begin{flushright}$\square$\end{flushright}

\bibliographystyle{plain}
\bibliography{bibliothese}


\end{document}